\documentclass[12pt,twoside]{amsart}

\usepackage{amssymb}
\usepackage{graphicx}
\usepackage[pdfauthor={David Cook II},
            pdftitle={The uniform face ideals of a simplicial complex},
            pdfsubject={Combinatorial commutative algebra},
            pdfkeywords={Monomial ideal, linear resolution, cellular resolution, Betti numbers, simplicial complex, vertex colouring, face ideal},
            pdfproducer={LaTeX with hyperref},
            pdfcreator={latex->dvips->ps2pdf},
            pdfpagemode=UseOutlines,
            bookmarksopen=true,
            bookmarksopenlevel=2,
            letterpaper,
            bookmarksnumbered=true,
            pdfborder={0 0 0},
            colorlinks=false]{hyperref}
\usepackage[margin=1in]{geometry}
\usepackage[foot]{amsaddr}

\def\urltilda{\kern -.15em\lower .7ex\hbox{\~{}}\kern .04em}  
\newcommand{\NN}{\ensuremath{\mathbb{N}}}                     
\newcommand{\ZZ}{\ensuremath{\mathbb{Z}}}                     
\newcommand{\mm}{\ensuremath{\mathfrak{m}}}                   
\newcommand{\mC}{\ensuremath{\mathcal{C}}}                    
\newcommand{\mD}{\ensuremath{\mathcal{D}}}                    
\newcommand{\mE}{\ensuremath{\mathcal{E}}}                    
\newcommand{\mF}{\ensuremath{\mathcal{F}}}                    
\newcommand{\mN}{\ensuremath{\mathcal{N}}}                    
\newcommand{\mS}{\ensuremath{\mathcal{S}}}                    
\newcommand{\st}{\ensuremath{\colon\,}}                       
\newcommand{\dcup}{\ensuremath{\,\mathaccent\cdot\cup\,}}     
\newcommand{\ddd}{\ensuremath{\dcup \cdots \dcup}}            
\newcommand{\eps}{\ensuremath{\varepsilon}}                   
\DeclareMathOperator{\ass}{ass}                               
\DeclareMathOperator{\cl}{cl}                                 
\DeclareMathOperator{\codim}{codim}                           
\DeclareMathOperator{\depth}{depth}                           
\DeclareMathOperator{\ind}{ind}                               
\DeclareMathOperator{\lcm}{lcm}                               
\DeclareMathOperator{\link}{link}                             
\DeclareMathOperator{\pdim}{pdim}                             
\DeclareMathOperator{\reg}{reg}                               
\DeclareMathOperator{\sgn}{sgn}                               

\numberwithin{figure}{section}
\numberwithin{equation}{section}

\newtheorem{theorem}{Theorem}[section]
\newtheorem{lemma}[theorem]{Lemma}
\newtheorem{proposition}[theorem]{Proposition}
\newtheorem{corollary}[theorem]{Corollary}

\theoremstyle{definition}
\newtheorem{definition}[theorem]{Definition}
\newtheorem{construction}[theorem]{Construction}
\newtheorem{remark}[theorem]{Remark}
\newtheorem{example}[theorem]{Example}

\newtheorem*{acknowledgement}{Acknowledgement}

\begin{document}

\title[Uniform face ideals]{The uniform face ideals of a simplicial complex}
\author[D.\ Cook II]{David Cook II}
\address{Department of Mathematics, University of Notre Dame, Notre Dame, IN 46556, USA}
\email{\href{mailto:dcook8@nd.edu}{dcook8@nd.edu}}
\subjclass[2010]{13F55, 05E45, 13D02, 05C15, 06A12}
\keywords{Monomial ideal, linear resolution, cellular resolution, Betti numbers, simplicial complex, vertex colouring, face ideal}

\begin{abstract}
    We define the uniform face ideal of a simplicial complex with respect to an ordered proper vertex colouring of the
    complex.  This ideal is a monomial ideal which is generally not squarefree.  We show that such a monomial ideal has a
    linear resolution, as do all of its powers, if and only if the colouring satisfies a certain nesting property.

    In the case when the colouring is nested, we give a minimal cellular resolution supported on a cubical complex.  From this,
    we give the graded Betti numbers in terms of the face-vector of the underlying simplicial complex.  Moreover, we explicitly
    describe the Boij-S\"oderberg decompositions of both the ideal and its quotient.  We also give explicit formul\ae\ for
    the codimension, Krull dimension, multiplicity, projective dimension, depth, and regularity.  Further still, we describe
    the associated primes, and we show that they are persistent.
\end{abstract}

\maketitle


\section{Introduction}\label{sec:introduction}

One method of generating ideals with specific properties, e.g., linear resolutions or persistent associated primes,
is to construct an ideal from a combinatorial object which has structure that can be exploited to force the desired properties on the ideal.
A classical approach to generating ideals has been the Stanley-Reisner correspondence that associates to a simplicial
complex $\Delta$ on $n$ vertices a squarefree monomial ideal $I_\Delta$ in an $n$-variate polynomial ring
$R = K[x_1,\ldots,x_n]$ (see, e.g., the books of Herzog and Hibi~\cite{HH}, Miller and Sturmfels~\cite{MS}, and
Stanley~\cite{St}).  For example, Eagon and Reiner~\cite[Corollary~5]{ER} showed that if the Alexander dual
$\Delta^{\vee}$ of the simplicial complex is pure and shellable, both of which are combinatorial conditions, then the
associated Stanley-Reisner ring $R/I_\Delta$ has a linear resolution and the Betti numbers can be derived from the
$h$-vector of $\Delta^{\vee}$.

In an alternate use of squarefree ideals, Herzog and Hibi~\cite{HH-p} studied the Hibi ideal $H_P$ of a poset $P$.  The
Hibi ideal is in a polynomial ring with two variables for each element:  one each to encode the presence and absence
of the element from an order ideal.  They showed that every power of a Hibi ideal has a linear resolution, and the Betti
numbers and the primary decomposition of the Hibi ideal can be described in terms of structural properties of the poset.

More recently, Biermann and Van Tuyl~\cite{BVT} constructed a new simplicial complex $\Delta_\mC$ from a given
simplicial complex $\Delta$ and proper vertex colouring $\mC$ of $\Delta$.  This new complex is pure and
vertex-decomposable (hence shellable), and its $h$-vector is the $f$-vector of the original complex.  Hence using
Eagon and Reiner's aforementioned result, they showed that $I_{\Delta_{\mC}^{\vee}}$ has a linear resolution with Betti numbers
derived from the $f$-vector of $\Delta$ independently from the colouring $\mC$.  As it turns out (see Section~\ref{sub:singleton}),
if the colouring is the collection of singletons of the vertices of $\Delta$, then the ideal $I_{\Delta_{\mC}^{\vee}}$ is in
a polynomial ring with two variables for each vertex:  one each to encode the presence and absence of the vertex
from a face of $\Delta$.

Given a simplicial complex $\Delta$ and an ordered proper vertex colouring $\mC$ of $\Delta$, we follow the approach of using
two variables to encode the presence and absence of an element from a given subset to define the \emph{uniform face ideal of
$\Delta$ with respect to $\mC$}, which we denote $I(\Delta, \mC)$; see Definition~\ref{def:ufi}.  While $I(\Delta, \mC)$ is generated
in a single degree, as in the above cases, it is not generally a squarefree ideal, contrary to the above cases.  However,
$I(\Delta, \mC)$ is squarefree precisely when the colouring is the collection of singletons of the vertices of $\Delta$; in this
case, $I(\Delta, \mC)$ is one of the ideals studied by Biermann and Van Tuyl~\cite{BVT}.

This manuscript is devoted to the study of the family of uniform face ideals.  In particular, we show that the presence of several
ideal properties are equivalent to the colouring being of a special type.  We define a \emph{nested proper vertex colouring}
to be an ordered colouring so that the links of the vertices of a given colour are linearly ordered by containment
(Definition~\ref{def:nested}); this is a simplicial analogue of the graph theoretic concept defined by the author~\cite{Co}.  Francisco,
Mermin, and Schweig~\cite{FMS} defined the $Q$-Borel property for a poset $Q$, which is a generalisation of the Borel property.
We define a poset $Q_k$ such that $I(\Delta, \mC)$ is $Q_k$-Borel precisely when $\mC$ is a nested $k$-colouring (Theorem~\ref{thm:Q-Borel}).
Using this connection, we show that the product of two uniform face ideals coming from nested colourings is also a uniform face
ideal coming from a nested colouring (Corollary~\ref{cor:nested-ufi-products}).

One desired property that an ideal generated in a single degree may enjoy is having a linear resolution, that is, all minimal syzygies
are linear.  Biermann and Van Tuyl~\cite{BVT}, Corso and Nagel~\cite{CN-2008} and \cite{CN-2009}, Herzog and Hibi~\cite{HH-p}, Nagel and Reiner~\cite{NR},
and Nagel and Sturgeon~\cite{NS} each studied squarefree monomial ideals associated to some combinatorial structure (simplicial complexes, posets,
or Ferrers hypergraphs) that have linear resolutions.  Indeed, the ideals studied by Herzog and Hibi have the additional property that their powers
always have linear resolutions.  This need not always happen; indeed, Sturmfels~\cite{Stu} gave an example of a squarefree monomial ideal in six
variables that has a linear resolution but whose second power does not.  We show here that the uniform face ideal $I(\Delta, \mC)$
has a linear resolution precisely when $\mC$ is a nested colouring.  Since products of uniform face ideals coming from nested colourings
are also uniform face ideals coming from nested colourings, all powers of uniform face ideals coming from nested colourings also have linear
resolutions (Theorem~\ref{thm:ufi-linear-nested}).

In each of~\cite{BVT}, \cite{HH-p}, \cite{NR}, and \cite{NS}, the $\ZZ$-graded---and hence total---Betti numbers of the studied
squarefree monomial ideal can be derived from underlying properties of the associated combinatorial structure.  We similarly
describe the $\ZZ$-graded Betti numbers of all uniform face ideals coming from nested colourings using the $f$-vector of the
simplicial complex (Theorem~\ref{thm:betti}).  Nagel and Sturgeon also gave the Boij-S\"oderberg decomposition (see, e.g., \cite{BS-2012})
of the Betti tables of both the ideal they studied and its quotient.  Using this, they classified the Betti tables possible for
ideals with a linear resolution (see also~\cite{Mu}).  In our case, we provide the explicit Boij-S\"oderberg
decomposition---with combinatorial interpretations of the coefficients---for the Betti table of both $I(\Delta, \mC)$
(Proposition~\ref{pro:bsd-I}) and $R/I(\Delta, \mC)$(Proposition~\ref{pro:bsd-RI}).

Another useful property that an ideal may enjoy is having a minimal free resolution supported on a CW-complex, as described
by Bayer and Sturmfels~\cite{BS}.  Velasco~\cite{Ve} showed that not all monomial ideals have such a resolution.  Despite this,
many classes of monomial ideals do have cellular resolutions.  Indeed, Nagel and Reiner~\cite{NR} described a cellular resolution
of several ideals associated to Ferrers hypergraphs; with a similar approach, Dochtermann and Engstr\"om~\cite{DE} described a
cellular resolution of edge ideals of cointerval hypergraphs.  Recently, Engstr\"om and Nor\'en~\cite{EN} gave cellular resolutions
for all powers of certain edge ideals.  We give a cellular resolution for a uniform face ideal coming from a nested colouring
that is supported on a cubical complex (Theorem~\ref{thm:cellular-resolution}).

Moreover, it is desirable for an ideal to have persistent associated primes, that is, $\ass(I^{i}) \subset \ass(I^{i+1})$ for $i \geq 1$.
Francisco, H\`a, and Van Tuyl~\cite{FHVT} described the associated primes of the cover ideals of graphs.  Using this, they showed
that the cover ideals of perfect graphs have persistent associated primes.  Mart\'inez-Bernal, Morey, and Villarreal~\cite{MBMV}
showed that all edge ideals of graphs (i.e., quadratically generated squarefree monomial ideals) have persistent primes.  Further,
Bhat, Biermann, and Van Tuyl~\cite{BBVT} described a new family of squarefree monomial ideals, the partial $t$-cover ideal of a graph,
which also have persistent associated primes.  In each case, the persistence is established via exploitation of the underlying combinatorial
structure.  We similarly exploit the structure of a nested colouring to describe the associated primes of a uniform face ideal coming from
a nested colouring (Corollary~\ref{cor:assoc}), and thus show that they have persistent associated primes (Theorem~\ref{thm:persistence}).

The remainder of the manuscript is organised as follows.  In Section~\ref{sec:prelim}, we recall the necessary algebraic and combinatorial
definitions.  In Section~\ref{sec:colouring}, we describe nested colourings of simplicial complexes, as well as the connection
to the graph theoretic analogue.  In Section~\ref{sec:ufi}, we define the uniform face ideal, and we consider
the special case of the singleton colouring.  In Section~\ref{sec:exchange}, we classify the presence and absence of various exchange
properties of ideals.  In Section~\ref{sec:syzygies}, we study the first syzygies of a uniform face ideal.  For the remainder of the
manuscript, we consider only nested colourings.  In Section~\ref{sec:resolution}, we give a minimal cellular resolution of a uniform face
ideal, and we also give the $\ZZ$-graded Betti numbers.  In Section~\ref{sec:ferrers}, we demonstrate the Boij-S\"oderberg
decompositions of both a uniform face ideal and its quotient.  In the final section, Section~\ref{sec:properties}, we classify several
algebraic properties of $R/I(\Delta, \mC)$, including giving explicit formul\ae\ for the codimension, Krull dimension, multiplicity, projective
dimension, depth, and regularity.  We also consider the Cohen-Macaulay and unmixed properties, and we describe the associated primes, which
we show to be persistent.

\section{Preliminaries}\label{sec:prelim}

In this section, we recall the necessary algebraic and combinatorial background that will be used throughout this manuscript.

\subsection{Free resolutions \& derivative algebraic invariants}\label{sub:prelim-res}~

Let $R = K[x_1,\ldots,x_n]$ be the $n$-variate polynomial ring over a field $K$, and let $M$ be a finitely
generated graded $R$-module, e.g., a homogeneous ideal $I$ of $R$ or its quotient $R/I$.  The $R$-module $M(d)$
given by the relations $[M(d)]_i := [M]_{i+d}$ is the \emph{$d^{\rm th}$ twist of $M$}.  A \emph{minimal free resolution of $M$}
is an exact sequence of free $R$-modules of the form
\[
    0 \longrightarrow F_n \longrightarrow \cdots \longrightarrow F_0 \longrightarrow M \longrightarrow 0,
\]
where $F_i$, $0 \leq i \leq n$, is the free $R$-module
\[
    \bigoplus_{j \in \ZZ} R(-j)^{\beta_{i,j}(M)}.
\]
The numbers $\beta_{i,j}(M)$ are the \emph{graded Betti numbers of $M$}, and $\beta_i(M) = \sum_{j \in \ZZ}\beta_{i,j}(M)$
is the \emph{$i^{\rm th}$ total Betti number of $M$}.  The $R$-module $M$ has a \emph{$d$-linear resolution}
if $\beta_{i,j}(M) = 0$ for all $0 \leq i \leq n$ and $j \neq i + d$.

The \emph{Poincar\'e polynomial of $M$} is $p_{M}(t) = \sum_{i \geq 0} \sum_{j \geq i} (-1)^i \beta_{i,j}(M) t^j$,
and the \emph{Hilbert series of $M$} is the formal power series $H_{M}(t) = p_{M}(t) / (1-t)^n$.  If we write
$H_{M(t)} = Q_{M}(t) / (1-t)^d$, where $Q_{M}(1) \neq 0$, then $Q_{M}(t)$ is the \emph{$Q$-polynomial of $M$}
and $d = \dim{M}$ is the \emph{Krull dimension of $M$}.  The \emph{codimension of $M$} is $\codim{M} = \dim{R} - \dim{M}$.
Finally, the \emph{multiplicity} (or \emph{degree}) \emph{of $M$} is $e(M) = Q_{M}(1)$.

The \emph{regularity of $M$} is $\reg{M} = \max\{ j - i  \st \beta_{i,j}(M) \neq 0 \}$.
The \emph{projective dimension of $M$} is $\pdim{M} = \max\{ i \st \beta_{i,j}(M) \neq 0 \mbox{~for some~} j \}$,
that is, $\pdim{M}$ is the length of the minimal free resolution of $M$.  The \emph{depth of $M$} is the length of the longest
homogeneous $M$-sequence, denoted $\depth{M}$.  By the Auslander-Buchsbaum formula, $\pdim{M} + \depth{M} = \pdim{R}$, we may
more easily compute the depth of $M$.  The $R$-module $M$ is \emph{Cohen-Macaulay} if $\depth{M} = \dim{M}$.

\subsection{Simplicial complexes}\label{sub:prelim-simpcomp}~

For a positive integer $n$, we define $[n]$ to be the set $\{1, \ldots, n\}$.  A \emph{simplicial complex} $\Delta$ on
the vertex set $[n]$ is a collection of subsets of $[n]$ closed under inclusion.  The elements of $\Delta$ are
\emph{faces}, and the maximal faces are \emph{facets}.  The \emph{dimension of a face $\sigma$} is
$\dim \sigma := \#\sigma - 1$, and the \emph{dimension of $\Delta$} is the maximum dimension $\dim \Delta$ of its faces.
The \emph{$f$-vector} (or \emph{face vector}) \emph{of $\Delta$} is the $(d+1)$-tuple $f(\Delta) = (f_{-1}, \ldots, f_{d-1})$,
where $d = \dim \Delta + 1$ and $f_i$ is the number of faces of dimension $i$ in $\Delta$.

The \emph{Alexander dual} of a simplicial complex $\Delta$ on $[n]$ is the simplicial complex $\Delta^\vee$ on $[n]$
with faces $[n] \setminus \sigma$, where $\sigma \notin \Delta$.  For any face $\sigma$ of $\Delta$, the
\emph{link of $\sigma$ in $\Delta$} is the simplicial complex
\[
    \link_{\Delta}(\sigma) := \{ \tau \in \Delta \st \sigma \cup \tau \in \Delta, \sigma \cap \tau = \emptyset\}.
\]

Let $\Delta$ be a simplicial complex on $[n]$.  The \emph{Stanley-Reisner ideal of $\Delta$} is the ideal
$I_\Delta := (x^\sigma \st \sigma \notin \Delta)$ of $R = K[x_1, \ldots, x_n]$, where $x^\sigma = \prod_{i \in \sigma} x_i$.
It is well-known that the Stanley-Reisner ideals are precisely the squarefree monomial ideals.  The quotient
ring $K[\Delta] := R/I_{\Delta}$ is the \emph{Stanley-Reisner ring of $\Delta$}.

\subsection{Graphs}\label{sub:prelim-graphs}~

A \emph{simple graph} $G$ is a vertex set $V(G)$, e.g., $[n]$, together with a collection of edges $E(G)$, i.e., $2$-subsets of $V(G)$.  Two
vertices $u$ and $v$ of $G$ are \emph{adjacent} if $\{u, v\}$ is an edge of $G$.  The \emph{neighbourhood of $v$ in $G$}
is the set $N_G(v)$ of vertices adjacent to $v$ in $G$.  A subset $A$ of the vertices of $G$ is an \emph{independent set of $G$}
if no two vertices in $A$ are adjacent.  Similarly, $A$ is a \emph{clique of $G$} if every pair of vertices in $A$ are adjacent.

A partition $\mC = C_1 \ddd C_k$ of $[n]$ is a \emph{proper vertex $k$-colouring} of $G$ if the $C_i$ are independent
sets of $G$; in this case, the $C_i$ are the \emph{colour classes} of $\mC$.  The \emph{chromatic number of $G$}
$\chi(G)$ is the least integer $k$ such that a proper vertex $k$-colouring of $G$ exists.

There is a natural graph associated to a simplicial complex $\Delta$ on $[n]$.  The \emph{underlying graph} (or \emph{$1$-skeleton})
\emph{of $\Delta$} is the simple graph $G_{\Delta}$ on vertex set $[n]$ with an edge between $i$ and $j$ if and only if
$\{i, j\} \in \Delta$.  Moreover, there are a pair of natural simplicial complexes associated to a simple graph $G$.
The \emph{clique complex of $G$} is the simplicial complex $\cl{G}$ of cliques of $G$, and the \emph{independence complex of $G$}
is the simplicial complex $\ind{G}$ of independent sets of $G$.  Notice that the minimal non-faces of $\cl{G}$ are the non-edges
of $G$, and further the minimal non-faces of $\ind{G}$ are the edges of $G$.  Hence the independence complex of $G$ is the
clique complex of the complement of $G$.  Moreover, the Stanley-Reisner ideal of a clique (independence) complex is generated by
quadrics; these are precisely the flag complexes.  We also note that $G_{\cl{G}} = G$ for any graph $G$.

\subsection{Posets}\label{sub:prelim-posets}~

A \emph{poset} is a set $P$ endowed with a \emph{partial order} $\leq$ that is antisymmetric, reflexive, and transitive.
A relation $u < v$ in $P$ is a \emph{covering relation} if $u \leq w \leq v$ implies either $u = w$ or $w = v$.
For any two elements $u$ and $v$ of $P$, an element $w$ of $P$ is the \emph{meet of $u$ and $v$} if it is the unique
element of $P$ such that $w < u$ and $w < v$.  Further, $P$ is a \emph{meet-semilattice} if every pair of elements of
$P$ has a meet.

A subposet $Q$ of $P$ is an \emph{order ideal of $P$} if $q \in Q$ and $p \leq q$ in $P$ implies $p \in Q$.
The \emph{interval of $u$ and $v$ in $P$} is the subposet $[u, v] = \{w \in P \st u \leq w \leq v\}$.
For any $n \in \NN$, the \emph{$n$-boolean poset} (or \emph{$n$-boolean lattice}) is the poset $B_n$ of all
subsets of $[n]$ ordered by inclusion.  A meet-semilattice $P$ is \emph{meet-distributive} if every interval
$[u, v]$ of $P$ such that $u$ is the meet of the $t$ elements of $[u,v]$ covered by $v$ is isomorphic to the $t$-boolean poset.

\section{Colouring simplicial complexes}\label{sec:colouring}

In this section, we define a colouring of a simplicial complex, and further define a special family
of proper vertex colourings.

\subsection{Proper vertex colourings}\label{sub:proper}~

The following concepts are simplicial analogues of some common graph theoretic concepts.
A set $A \subset [n]$ is an \emph{independent set of $\Delta$} if no two members of $A$ are in a common face of $\Delta$.
A partition $\mC = C_1 \ddd C_k$ of $[n]$ is a \emph{proper vertex $k$-colouring} of $\Delta$ if the $C_i$ are independent
sets of $\Delta$; in this case, the $C_i$ are the \emph{colour classes} of $\mC$.  The \emph{chromatic number of $\Delta$}
$\chi(\Delta)$ is the least integer $k$ such that a proper vertex $k$-colouring of $\Delta$ exists.

It is clear that the proper vertex colourings of $\Delta$ are precisely the proper vertex colourings of $G_\Delta$.

\begin{lemma}\label{lem:chi}
    Let $\Delta$ be a simplicial complex on $[n]$.  If $\mC$ is a partition of $[n]$,
    then $\mC$ is a proper vertex colouring of $\Delta$ if and only if it is a proper vertex colouring
    of $G_\Delta$.

    In particular, $\chi(\Delta) = \chi(G_\Delta)$.
\end{lemma}
\begin{proof}
    This follows immediately as the edges of $G_\Delta$ are precisely the $1$-faces of $\Delta$.
\end{proof}

\subsection{Nested colourings}\label{sub:nested}~

As defined and studied in~\cite{Co}, an independent set $A$ of a finite simple graph $G$ is \emph{nested}
if the vertices of $A$ can be linearly ordered so that $v \leq u$ implies $N_G(u) \subset N_G(v)$; such an
order is a \emph{nesting order} of $A$.  A proper vertex colouring $\mC$ of $G$ is \emph{nested} if every
colour class of $\mC$ is nested.  The \emph{nested chromatic number} $\chi_N(G)$ is the least integer $k$
such that a nested $k$-colouring of $G$ exists.

We define here the simplicial analogue of a nested colouring.  The link of a vertex in a simplicial complex
will play the role of the neighbourhood of a vertex in a graph.

\begin{definition}\label{def:nested}
    Let $\Delta$ be a simplicial complex, and let $u, v$ be vertices of $\Delta$.  An independent set $A$ of $\Delta$
    is \emph{nested} if the vertices of $A$ can be linearly ordered so that $v \leq u$ implies
    $\link_{\Delta}(u) \subset \link_{\Delta}(v)$; such an order is a \emph{nesting order} of $A$.  A proper vertex
    $k$-colouring $\mC = C_1 \ddd C_k$ of $\Delta$ is \emph{nested} if every colour class of $\mC$ is nested.
    The \emph{nested chromatic number} $\chi_N(\Delta)$ is the least integer $k$ such that a nested $k$-colouring of
    $\Delta$ exists.
\end{definition}

The nesting orders on the vertices of a nested independent set are the same, up to permutations of vertices that have
precisely the same link.  Such vertices are indistinguishable except for their label.

\begin{example}\label{exa:nested}
    Let $\Delta = \langle abc, bcd, ce, de, df \rangle$.  There are $20$ proper vertex colourings of $\Delta$
    of which $14$ are nested.  For example, $\mC = \{d,a\} \dcup \{b,e\} \dcup \{c,f\}$ is a nested colouring of
    $\Delta$; thus $\chi_N(\Delta) = 3$.  However, $\mD = \{a,f\} \dcup \{b,e\} \dcup \{c\} \dcup \{d\}$ is a
    non-nested colouring of $\Delta$ as $\link_{\Delta}(f) = \langle d \rangle \not\subset \link_{\Delta}(a) = \langle bc \rangle$.
    See Figure~\ref{fig:nested} for illustrations.
    \begin{figure}[!ht]
        \begin{minipage}[b]{0.32\linewidth}
            \centering
            \includegraphics[scale=1.75]{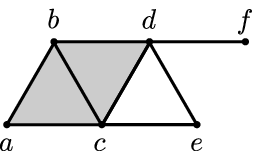}\\
            \emph{(i) The complex $\Delta$.}
        \end{minipage}
        \begin{minipage}[b]{0.32\linewidth}
            \centering
            \includegraphics[scale=1.75]{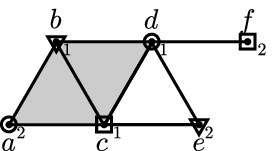}\\
            \emph{(ii) A nested colouring.}
        \end{minipage}
        \begin{minipage}[b]{0.32\linewidth}
            \centering
            \includegraphics[scale=1.75]{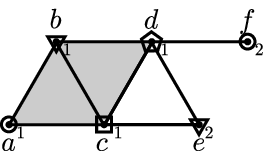}\\
            \emph{(iii) A non-nested colouring.}
        \end{minipage}
        \caption{The simplicial complex $\Delta$ has nested chromatic number $3$.}
        \label{fig:nested}
    \end{figure}
\end{example}

Exchanging a vertex of a face for a lesser (in the nesting order) vertex of the same colour generates another face of the simplicial complex.
This gives an alternate, and perhaps more useful, condition on a partition of the vertices that is equivalent to being a
nested colouring.

\begin{proposition}\label{pro:alt-nested}
    Let $\Delta$ be a simplicial complex on $[n]$.  If $\mC = C_1 \ddd C_k$ is a partition of $[n]$, then $\mC$ is a nested
    $k$-colouring if and only if there is an ordering on the vertices of each class $C_i$ such that if $v$ is less than $u$
    in that order, and $\sigma$ is a face of $\Delta$ containing $u$, then $(\sigma \dcup \{v\}) \setminus \{u\}$ is a face
    of $\Delta$.
\end{proposition}
\begin{proof}
    This follows immediately since $\sigma$ is a face of $\Delta$ containing $u$ if and only if $\sigma \setminus \{u\}$ is
    a face of $\link_{\Delta}(u)$.
\end{proof}

Nested colourings of $\Delta$ are nested colourings of $G_\Delta$; the converse holds when $\Delta$ is the flag.

\begin{lemma}\label{lem:chi_N}
    Let $\Delta$ be a simplicial complex on $[n]$.  If $\mC = C_1 \ddd C_k$ is a nested colouring
    of $\Delta$, then $\mC$ is a nested colouring of $G_\Delta$, and the converse holds when $\Delta$ is flag.

    In particular, $\chi_N(\Delta) \geq \chi_N(G_\Delta)$, and equality holds when $\Delta$ is flag.
\end{lemma}
\begin{proof}
    By Proposition~\ref{pro:alt-nested}, if $v < u$ in $C_i$, and $\{u, w\}$ is a face of $\Delta$, then $\{v, w\}$ is a
    face of $\Delta$.  In particular, if $v < u$ in $C_i$, and $\{u, w\}$ is an edge of $G_\Delta$, then $\{v, w\}$ is an
    edge of $G_\Delta$.  By~\cite[Proposition~2.16]{Co}, the latter is equivalent to $\mC$ being a nested colouring of
    $G_\Delta$.

    Furthermore, if $\Delta$ is flag, i.e., $\Delta = \cl{G_\Delta}$, then the faces of $\Delta$ are precisely the cliques
    of $G_\Delta$.  Thus the concept of a nested colouring is the same for both $\Delta$ and $G_\Delta$, again using
    Proposition~\ref{pro:alt-nested} and~\cite[Proposition~2.16]{Co}.
\end{proof}

There exist nested colourings of $G_\Delta$ that are not nested colourings of $\Delta$, and the nested chromatic numbers
need not be the same.

\begin{example}\label{exa:non-nested}
    First, recall $\Delta = \langle abc, bcd, ce, de, df \rangle$ given in Example~\ref{exa:nested}.  This complex is
    not flag, as $cde \notin \Delta$.  However, the nested colourings of $\Delta$ are precisely the nested colourings
    of $G_\Delta$.

    Now, let $\Gamma = \langle abc, bd, cde \rangle$.  In this case, $\chi_N(\Gamma) = 5$ but $\chi_N(G_\Gamma) = 3$.
    In particular, $\mC = \{d,a\} \dcup \{b,e\} \dcup \{c\}$ is a nested colouring of $G_\Delta$.
    See Figure~\ref{fig:non-nested} for illustrations.
    \begin{figure}[!ht]
        \begin{minipage}[b]{0.48\linewidth}
            \centering
            \includegraphics[scale=1.75]{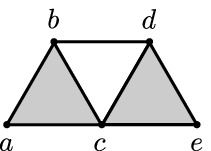}\\
            \emph{(i) The complex $\Gamma$.}
        \end{minipage}
        \begin{minipage}[b]{0.48\linewidth}
            \centering
            \includegraphics[scale=1.75]{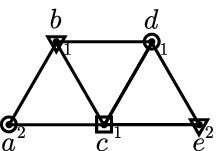}\\
            \emph{(ii) A nested colouring of $G_{\Gamma}$.}
        \end{minipage}
        \caption{A simplicial complex with a larger nested chromatic number than its underlying graph.}
        \label{fig:non-nested}
    \end{figure}
\end{example}

\begin{remark}\label{rem:nested-analogues}
    Most of the results about nested colourings of graphs in~\cite{Co} have analogues for nested colourings of
    simplicial complexes.  For instance, Proposition~\ref{pro:alt-nested} is the simplicial analogue
    of~\cite[Proposition~2.16]{Co}.  Moreover, since the nested chromatic number of a simplicial complex can
    be computed as the Dilworth number of a poset, as in~\cite[Corollary~2.19]{Co}, we see that computing the
    nested chromatic number of a simplicial complex can be done in polynomial time; see~\cite[Theorem~2.21]{Co}
    for the graph theoretic analogue.
\end{remark}

\section{The uniform face ideals}\label{sec:ufi}

In this section, we introduce the objects of interest in this manuscript:  the uniform face ideal of a simplicial complex
with respect to a proper vertex colouring of the complex.  We further study a specific nested colouring that has
many nice properties.

\subsection{The uniform face ideals}\label{sub:ufi}~

A colouring $\mC$ of a simplicial complex $\Delta$ is \emph{ordered} if each colour class is endowed with a linear
order on the vertices in the colour class.  In this case, we may use a vector to identify the faces of a simplicial
complex.

\begin{definition}\label{def:index-vector}
    Let $\Delta$ be a simplicial complex, and let $\mC = C_1 \ddd C_k$ be an ordered proper vertex $k$-colouring
    of $\Delta$.  For a face $\sigma$ of $\Delta$, the \emph{index vector of $\sigma$ with respect to $\mC$} is
    the vector $e(\sigma) = (e_1(\sigma), \ldots, e_k(\sigma))$, where $e_i(\sigma) = 0$ if
    $\sigma \cap C_i = \emptyset$ and $e_i(\sigma) = j$ if $\sigma \cap C_i = \{v_j\}$.
\end{definition}

The index vector of $\emptyset$ with respect to any colouring is the zero vector, and the index vectors with
precisely one nonzero entry are the index vectors of vertices.

We are now ready to define the object of interest in this manuscript.

\begin{definition}\label{def:ufi}
    Let $\Delta$ be a simplicial complex on $[n]$, and let $\mC = C_1 \ddd C_k$ be an ordered
    proper vertex $k$-colouring of $\Delta$.  Let $R = K[x_1, \ldots, x_k, y_1, \ldots, y_k]$
    be the polynomial ring in $2k$ variables over the field $K$.  For any face $\sigma \in \Delta$,
    the \emph{uniform monomial of $\sigma$ with respect to $\mC$} is the monomial
    \[
        m_\sigma = \prod_{i = 1}^k x_i^{\#C_i - e_i(\sigma)} y_i^{e_i(\sigma)}
    \]
    in $R$, where $e(\sigma)$ is the index vector of $\sigma$ with respect to $\mC$.
    Further, the \emph{uniform face ideal of $\Delta$ with respect to $\mC$} is the $R$-ideal
    \[
        I(\Delta, \mC) = (m_\sigma \st \sigma \in \Delta).
    \]
\end{definition}

The uniform monomial of $\emptyset$ with respect to any colouring is $x_1 \cdots x_k$.
Clearly, $2\chi(\Delta)$ is a lower bound for the number of necessary variables to construct a uniform
face ideal of $\Delta$, and empty colour classes have no effect on the generators of $I(\Delta, \mC)$
but do increase the number of variables in the polynomial ring.

\begin{example}\label{exa:ufi}
    Let $\Delta = \langle abc, bcd, ce, de, df \rangle$ as given in Example~\ref{exa:nested}.  Since
    $\Delta$ has a total of $17$ faces, the associated uniform face ideals will each have $17$ monomial generators.

    Consider the nested colouring $\mC = \{d,a\} \dcup \{b,e\} \dcup \{c,f\}$ of $\Delta$ given in Example~\ref{exa:nested}.
    Since $\mC$ is a $3$-colouring, we let $R = K[x_1, x_2, x_3, y_1, y_2, y_3]$, and further
    \begin{equation*}
        \begin{split}
            I(\Delta, \mC) = (&
                \underbrace{x_1^2  x_2^2  x_3^2 }_{\emptyset},
                \\ &
                \underbrace{y_1^2  x_2^2  x_3^2 }_{a},
                \underbrace{x_1^2  x_2y_2 x_3^2 }_{b},
                \underbrace{x_1^2  x_2^2  x_3y_3}_{c},
                \underbrace{x_1y_1 x_2^2  x_3^2 }_{d},
                \underbrace{x_1^2  y_2^2  x_3^2 }_{e},
                \underbrace{x_1^2  x_2^2  y_3^2 }_{f},
                \\ &
                \underbrace{y_1^2  x_2y_2 x_3^2 }_{ab},
                \underbrace{y_1^2  x_2^2  x_3y_3}_{ac},
                \underbrace{x_1^2  x_2y_2 x_3y_3}_{bc},
                \underbrace{x_1y_1 x_2y_2 x_3^2 }_{bd},
                \\ &\hspace{3em}
                \underbrace{x_1y_1 x_2^2  x_3y_3}_{cd},
                \underbrace{x_1^2  y_2^2  x_3y_3}_{ce},
                \underbrace{x_1y_1 y_2^2  x_3^2 }_{de},
                \underbrace{x_1y_1 x_2^2  y_3^2 }_{df},
                \\ &
                \underbrace{y_1^2  x_2y_2 x_3y_3}_{abc},
                \underbrace{x_1y_1  x_2y_2 x_3y_3}_{bcd}
            ).
        \end{split}
    \end{equation*}

    Further, recall the non-nested colouring $\mD = \{a,f\} \dcup \{b,e\} \dcup \{c\} \dcup \{d\}$, as given in
    Example~\ref{exa:nested}.  Since $\mD$ is a $4$-colouring, we consider $R = K[x_1, x_2, x_3, x_4, y_1, y_2, y_3, y_4]$.
    In this case, we have
    \begin{equation*}
        \begin{split}
            I(\Delta, \mD) = (&
                \underbrace{x_1^2  x_2^2  x_3 x_4}_{\emptyset},
                \\ &
                \underbrace{x_1y_1 x_2^2  x_3 x_4}_{a},
                \underbrace{x_1^2  x_2y_2 x_3 x_4}_{b},
                \underbrace{x_1^2  x_2^2  y_3 x_4}_{c},
                \underbrace{x_1^2  x_2^2  x_3 y_4}_{d},
                \underbrace{x_1^2  y_2^2  x_3 x_4}_{e},
                \underbrace{y_1^2  x_2^2  x_3 x_4}_{f},
                \\ &
                \underbrace{x_1y_1 x_2y_2 x_3 x_4}_{ab},
                \underbrace{x_1y_1 x_2^2  y_3 x_4}_{ac},
                \underbrace{x_1^2  x_2y_2 y_3 x_4}_{bc},
                \underbrace{x_1^2  x_2y_2 x_3 y_4}_{bd},
                \\ &\hspace{3em}
                \underbrace{x_1^2  x_2^2  y_3 y_4}_{cd},
                \underbrace{x_1^2  y_2^2  y_3 x_4}_{ce},
                \underbrace{x_1^2  y_2^2  x_3 y_4}_{de},
                \underbrace{y_1^2  x_2^2  x_3 y_4}_{df},
                \\ &
                \underbrace{x_1y_1 x_2y_2 y_3 x_4}_{abc},
                \underbrace{x_1^2  x_2y_2 y_3 y_4}_{bcd}
            ).
        \end{split}
    \end{equation*}
\end{example}

Recall that a pair of nesting orders on the vertices of a nested colouring differ only by permuting vertices
with the same link.  This implies that all nesting orders produce the same uniform face ideal, and so we will
henceforth refer to \emph{the} nesting order on the nested colouring.

\begin{lemma}\label{lem:ufi-nested}
    Let $\Delta$ be a simplicial complex, and let $\mC = C_1 \ddd C_k$ be a nested colouring of $\Delta$.
    If $\mC'$ and $\mC''$ are copies of $\mC$ endowed with (possibly distinct) nesting orders on the colour
    classes, then $I(\Delta, \mC') = I(\Delta, \mC'')$.
\end{lemma}
\begin{proof}
    Since the nesting order on each class is independent of the nesting orders of the other classes, we may
    assume the nesting orders on $\mC'$ and $\mC''$ only differ for $C_1$.  Moreover, since all nesting orders
    are the same up to the permutation of vertices with the same link, we may assume that the nesting order
    on $C_1$ differs only for two vertices $u$ and $v$ such that $\link_{\Delta}(u) = \link_{\Delta}(v)$ and
    no other vertex of $C_1$ is between $u$ and $v$ in both orders.

    Suppose $u$ and $v$ are the $j^{\rm th}$ and $(j+1)^{\rm st}$ vertices, respectively, in $C_1$ in $\mC$.
    Hence $u$ and $v$ are the $(j+1)^{\rm st}$ and $j^{\rm th}$ vertices, respectively, in $C_1$ in $\mC'$.
    Since $\{u\} \dcup \sigma \in \Delta$ if and only if $\{v\} \dcup \sigma \in \Delta$, for any face $\sigma$ of $\Delta$,
    we have that $(j, e_2, \ldots, e_k)$ is an index vector of a face in $\Delta$ under $\mC'$ if and only if
    $(j+1, e_2, \ldots, e_k)$ is an index vector of a face in $\Delta$ under $\mC'$.  Similarly, this holds if we
    replace $\mC'$ with $\mC''$.  As the uniform face ideals are constructed using the index vectors, and the
    set of index vectors for $\Delta$ with respect to $\mC'$ and $\mC''$ are the same, the uniform face ideals
    are also the same.
\end{proof}

Moreover, the product of uniform face ideals is again a uniform face ideal.   However, the resultant simplicial
complex depends on the colouring, ordering, and labeling of the factors.

\begin{proposition}\label{pro:ufi-product}
    Let $\Delta$ and $\Gamma$ be simplicial complexes.  If $\mC = C_1 \ddd C_k$ and $\mD = D_1 \ddd D_k$
    are proper vertex $k$-colourings of $\Delta$ and $\Gamma$, respectively, then
    $I(\Delta, \mC) \cdot I(\Gamma, \mD) = I(\Sigma, \mE)$ for some simplicial complex $\Sigma$
    and proper vertex $k$-colouring $\mE$ of $\Sigma$.
\end{proposition}
\begin{proof}
    Let $\mE = E_1 \ddd E_k$, where $E_i = \{v_{i,1}, \ldots, v_{i,\#C_k + \#D_k}\}$ for $1 \leq i \leq k$.

    Notice that for any simplicial complex $\Delta$ and proper vertex $k$-colouring $\mC$, by the construction
    of $I(\Delta, \mC)$, the sum of the exponents of $x_i$ and $y_i$ in $m$ is $\#C_i$, for every $i$ and every
    minimal generator $m$ of $I(\Delta, \mC)$.  Hence the sum of the exponents of $x_i$ and $y_i$ in $m_\sigma \cdot m_\tau$
    is $\#E_i = \#C_i + \#D_i$, for every $i$.  Thus $m_\sigma \cdot m_\tau = m_\rho$, where $\rho$
    is some subset of $E_1 \cup \cdots \cup E_k$.

    Let $\Sigma$ be the set of subsets $\rho$ of $E_1 \cup \cdots \cup E_k$ such that $\rho = m_\sigma \cdot m_\tau$,
    for some $\sigma \in \Delta$ and $\tau \in \Gamma$.  Clearly then, if $\Sigma$ is a simplicial complex, then
    $\mE$ is a proper vertex $k$-colouring of $\Sigma$ and $I(\Delta, \mC) \cdot I(\Gamma, \mD) = I(\Sigma, \mE)$.

    Since $x_1^{\#C_1} \cdots x_k^{\#C_k} \in I(\Delta, \mC)$ and $x_1^{\#D_1} \cdots x_k^{\#D_k} \in I(\Gamma, \mD)$,
    we have $\emptyset \in \Sigma$.  Let $\rho$ be any member of $\Sigma$ that is not empty, and suppose
    $m_\rho = m_\sigma \cdot m_\tau$, where $\sigma \in \Delta$ and $\tau \in \Gamma$.  Let $v$ be any vertex in $\rho$,
    and let $i$ be the index so that $v \in E_i$.  Since $\sigma \setminus C_i \in \Delta$ and $\tau \setminus D_i \in \Gamma$,
    $g = m_{\sigma \setminus C_i} \cdot m_{\tau \setminus D_i} \in I(\Delta, \mC) \cdot I(\Gamma, \mD)$.  Notice that $g$
    is $m_\rho$ with the exponent on $x_i$ changed to $\#E_i = \#C_i + \#D_i$ and the exponent on $y_i$ reduces to zero.
    That is, $m_\rho = m_{\rho \setminus \{v\}}$.  Thus $\rho \setminus \{v\} \in \Sigma$, and $\Sigma$ is
    closed under inclusion.  Therefore, $\Sigma$ is a simplicial complex.
\end{proof}

\begin{example}\label{exa:product}
    Let $\Delta = \langle abc, cd \rangle$ and $\Gamma = \langle ab, ac \rangle$ be simplicial complexes on $4$ and $3$ vertices, respectively,
    with nested colourings $\mC = \{a,d\} \dcup \{b\} \dcup \{c\}$ and $\mD = \{a\} \dcup \{b\} \dcup \{c\}$, respectively.  We then
    have $I(\Delta, \mC) \cdot I(\Gamma, \mD) = I(\Sigma, \mE)$, where $\Sigma = \langle adf, adf, aef, bdf, bdg, bef, cdf, cg \rangle$
    and $\mE = \{a,b,c\} \dcup \{d,e\} \dcup \{f,g\}$.  Note that $\mE$ is a nested colouring of $\Sigma$; in
    Corollary~\ref{cor:nested-ufi-products}, we show that this always occurs.
\end{example}

\subsection{The singleton colouring}\label{sub:singleton}~

Let $\Delta$ be a simplicial complex on $[n]$.  The proper vertex $n$-colouring of $\Delta$ with singleton
colour classes, that is, $\mS = \{1\} \ddd \{n\}$, is the \emph{singleton colouring of $\Delta$}.  Clearly,
the singleton colouring is a nested colouring.  In this case, the index vector $e$ of any face $\sigma$ of
$\Delta$ with respect to $\mS$ can be seen as the bit-vector encoding the presence of the vertices in $\sigma$.

The uniform face ideal $I(\Delta, \mC)$ is squarefree precisely when the colour classes of $\mC$ have cardinality
at most one, i.e., $\mC$ is $\mS$ together with empty colour classes.  Thus $I(\Delta, \mS)$ is the Stanley-Reisner
ideal of a simplicial complex, which turns out to be related to a simplicial complex previously studied.

Given a proper vertex colouring $\mC$ of a simplicial complex $\Delta$, Biermann and Van Tuyl~\cite{BVT} defined
a new simplicial complex with many nice properties.  We recall their construction here.

\begin{construction}{\cite[Construction~3]{BVT}}\label{con:BVT}
    Let $\Delta$ be a simplicial complex on $[n]$, and let $\mC = C_1 \ddd C_k$ be a proper
    vertex $k$-colouring of $\Delta$.  Define $\Delta_\mC$ to be the simplicial complex on
    $[n] \dcup \{1', \ldots, k'\}$ with faces $\sigma \cup \tau$, where $\sigma \in \Delta$ and
    $\tau$ is any subset of $\{1', \ldots, k'\}$ such that for all $j' \in \tau$ we have $\sigma \cap C_j = \emptyset$.
\end{construction}

The preceding construction was implicitly introduced by Bj\"orner, Frankl, and Stanley~\cite[Section~5]{BFS}.  It
was more recently introduced independently by Frohmader~\cite[Construction~7.1]{Fr}.  Biermann and Van Tuyl~\cite[Remark~4]{BVT}
have noted connections to related constructions.

Under an appropriate relabeling, $I(\Delta, \mS)$ is the Stanley-Reisner ideal of $\Delta_\mS^\vee$.

\begin{proposition}\label{pro:ufi-squarefree}
    Let $\Delta$ be a simplicial complex on $[n]$.  If $J$ is the Stanley-Reisner ideal of $\Delta_\mS^\vee$,
    where $i$ and $i'$ are associated to $x_i$ and $y_i$, respectively, for $i \in [n]$, then $J = I(\Delta, \mS)$.
\end{proposition}
\begin{proof}
    By construction, the minimal non-faces of $\Delta_\mS^\vee$ are precisely the complements of the facets of
    $\Delta_\mS$.  Further, by construction of $\Delta_\mS$, the facets are of the form $\sigma \cup \tau_\sigma$,
    where $\sigma \in \Delta$ and $\tau_\sigma = \{ i' \st 1 \leq i \leq n, i \notin \sigma \}$.  The complement of
    $\sigma \cup \tau_\sigma$ is $\overline{\sigma \cup \tau_\sigma} = \{ i' \st i \in \sigma\} \cup \{ i \st i' \notin \sigma\}$.
    Thus we have $e_i(\sigma) = 1$ if and only if $i'$ is in $\overline{\sigma \cup \tau_\sigma}$, and so the
    monomial associated to $\overline{\sigma \cup \tau_\sigma}$ in $J$ is
    \[
        x^{\overline{\sigma \cup \tau_\sigma}} = \prod_{i=1}^n x_i^{1-e_i(\sigma)} y_i^{e_i(\sigma)}.
    \]
    This is precisely the uniform monomial of $\sigma$ with respect to $\mS$.
\end{proof}

In~\cite[Theorem~13]{BVT}, $I_{\Delta_\mS^\vee} = I(\Delta, \mS)$ is shown to have a linear resolution with Betti numbers
easily described by the $f$-vector of $\Delta$, as we will see in Theorem~\ref{thm:betti}.
Moreover, a special case of Construction~\ref{con:BVT} was explored by the author with U.\ Nagel~\cite{CN} in the case of
flag complexes.  As will be seen in Corollary~\ref{cor:sqfree-unmixed}, the squarefree uniform face ideals that are unmixed
are precisely the ones coming from flag complexes.

\begin{example}\label{exa:ufi-squarefree}
    Let $\Delta = \langle abc, bcd, ce, de, df \rangle$ as given in Example~\ref{exa:nested},
    and consider the singleton colouring $\mS = \{a\} \ddd \{f\}$ of $\Delta$.  Set $R = K[x_1,\ldots,x_6,y_1,\ldots,y_6]$.
    Hence we have the squarefree monomial ideal (compare this to the ideals in Example~\ref{exa:ufi})
    \begin{equation*}
        \begin{split}
            I(\Delta, \mS) = (&
                \underbrace{x_1 x_2 x_3 x_4 x_5 x_6}_{\emptyset},
                \underbrace{y_1 x_2 x_3 x_4 x_5 x_6}_{a},
                \underbrace{x_1 y_2 x_3 x_4 x_5 x_6}_{b},
                \\ &
                \underbrace{x_1 x_2 y_3 x_4 x_5 x_6}_{c},
                \underbrace{x_1 x_2 x_3 y_4 x_5 x_6}_{d},
                \underbrace{x_1 x_2 x_3 x_4 y_5 x_6}_{e},
                \underbrace{x_1 x_2 x_3 x_4 x_5 y_6}_{f},
                \\ &
                \underbrace{y_1 y_2 x_3 x_4 x_5 x_6}_{ab},
                \underbrace{y_1 x_2 y_3 x_4 x_5 x_6}_{ac},
                \underbrace{x_1 y_2 y_3 x_4 x_5 x_6}_{bc},
                \underbrace{x_1 y_2 x_3 y_4 x_5 x_6}_{bd},
                \\ &\hspace{3em}
                \underbrace{x_1 x_2 y_3 y_4 x_5 x_6}_{cd},
                \underbrace{x_1 x_2 y_3 x_4 y_5 x_6}_{ce},
                \underbrace{x_1 x_2 x_3 y_4 y_5 x_6}_{de},
                \underbrace{x_1 x_2 x_3 y_4 x_5 y_6}_{df},
                \\ &
                \underbrace{y_1 y_2 y_3 x_4 x_5 x_6}_{abc},
                \underbrace{x_1 y_2 y_3 y_4 x_5 x_6}_{bcd}
            ).
        \end{split}
    \end{equation*}
    This ideal is the Stanley-Reisner ideal of the Alexander dual of $\Delta_\mS$, i.e., the Alexander dual of
    \[
        I_{\Delta_{\mS}} = (x_1y_1, x_2y_2, x_3y_3, x_4y_4, x_5y_5, x_6y_6, x_1x_4, x_1x_5, x_1x_6, x_2x_5, x_2x_6, x_3x_6, x_5x_6, x_3x_4x_5).
    \]
\end{example}

\begin{remark}\label{rem:olteanu}
    Olteanu~\cite{Ol} defined the \emph{monomial ideal of independent sets of a graph $G$} which is the Stanley-Reisner
    ideal of $(\ind{G})_{\mS}^{\vee}$.  Thus the monomial ideal of independent sets of a graph is the uniform face ideal
    $I(\ind{G}, \mS)$.  Olteanu~\cite[Corollary~2.3]{Ol} showed that this ideal always has a linear resolution and further
    gave explicit formul\ae\ for the regularity, Betti numbers, projective dimension, and Krull dimension.  Moreover, the
    presence of the Cohen-Macaulay property is classified therein.  We note that these results all correlate with the results
    given in Theorems~\ref{thm:ufi-linear-nested} and~\ref{thm:betti} and in Section~\ref{sub:derivative}.
\end{remark}

\section{Exchange properties of ideals}\label{sec:exchange}

In this section, we classify precisely when the uniform face ideal of a simplicial complex with respect to a
colouring has one of a variety of exchange properties of ideals.  Through this we see that nested colourings
force a nice structure on the uniform face ideal.

\subsection{Stable and strongly stable}\label{sub:stable}~

Let $R = K[x_1, \ldots, x_n]$, where $K$ is a field.  For a monomial $m \in R$, the \emph{maximum index of $m$}
$\mu(m)$ is the largest index $i$ such that $x_i$ divides $m$.  A monomial ideal $I$ is \emph{stable} if
$x_i m / x_{\mu(m)}$ is in $I$ for every monomial $m \in I$ and each $i < \mu(m)$.  Further still, a monomial ideal $I$
is \emph{strongly stable} if $x_i m / x_j$ is in $I$ for every monomial $m \in I$ and $i < j$ such that $x_j$ divides $m$.
Clearly, if $I$ is strongly stable, then $I$ is stable.  Furthermore, it suffices to only consider the monomials
$m$ that minimally generate $I$.

Eliahou and Kervaire~\cite{EK} proved stable ideals have minimal free resolutions that are very easy to describe.
Hulett~\cite{Hu} proved that the $\ZZ$-graded Betti numbers can be easily derived from the maximum indices of the minimal
generators.  In particular, if a stable ideal is generated by monomials in one degree, then it has a linear resolution.

A uniform face ideal is stable (and strongly stable) precisely when $\mC$ is a $1$-colouring.  Note that $1$-colourings
are always nested.

\begin{proposition}\label{pro:stable-ideal}
    Let $\Delta$ be a simplicial complex on $[n]$, and let $\mC$ be a $k$-colouring of $\Delta$
    without trivial colour classes.  The following statements are equivalent:
    \begin{enumerate}
        \item $I(\Delta, \mC)$ is stable,
        \item $I(\Delta, \mC)$ is strongly stable, and
        \item $\mC$ is a $1$-colouring.
    \end{enumerate}
\end{proposition}
\begin{proof}
    Clearly, condition (ii) implies condition (i).  Further, if $\mC$ is a $1$-colouring, then
    $I(\Delta, \mC) = (x_1^n, x_1^{(n-1)}y_1, \ldots, y_1^n) = (x,y)^n$, which is strongly stable.
    Thus condition (iii) implies condition (ii).

    Suppose condition (i) holds, i.e., $I(\Delta, \mC)$ is stable.  Without loss of generality,
    assume $x_1 < \ldots < x_k$.  Since $x_1^{\#C_1} \ldots x_k^{\#C_k}$ is in $I(\Delta, \mC)$,
    regardless of $\Delta$ and $\mC$, then $x_1^n \in I(\Delta, \mC)$.  By the structure of
    $I(\Delta, \mC)$, this implies that $k= 1$, i.e., condition (iii) holds.
\end{proof}

\subsection{\texorpdfstring{$Q$}{Q}-Borel}\label{sub:Q-Borel}~

Let $R = K[x_1, \ldots, x_n]$, where $K$ is a field.  A monomial ideal is \emph{Borel} if it is fixed under
the action of the Borel group.  If the characteristic of $K$ is zero, then a monomial ideal $I$ is strongly
stable if and only if it is Borel (see, e.g., \cite[Proposition~4.2.4]{HH}).  Hence the  Borel uniform face
ideals are precisely those which come from $1$-colourings.

Recently, Francisco, Mermin, and Schweig~\cite{FMS} generalised the Borel property using posets.
Let $Q$ be a poset on $\{x_1, \ldots, x_n\}$.  A monomial ideal in $R$ is \emph{$Q$-Borel} if $x_i m / x_j$ is in
$I$ for every monomial $m \in I$ and each $x_i < x_j$ in $Q$ such that $x_j$ divides $m$; the replacement of $m$ with
$x_i m / x_j$ is a \emph{$Q$-Borel move}.  Thus being Borel is equivalent to being $C_n$-Borel, where $C_n$ is
the $n$-chain poset $x_1 < \cdots < x_n$.  We note that it suffices to look only at the monomials $m$ that are
minimal generators of $I$.  A minimal generator $m$ of a monomial ideal $I$ is a \emph{$Q$-Borel generator} if it
is not generated from any other minimal generator of $I$ by a $Q$-Borel move.  Hence a monomial ideal $I$ is
\emph{principal $Q$-Borel} if it has a unique $Q$-Borel generator.

We identify a poset $Q$ for which the $Q$-Borel property distinguishes the uniform face ideals with respect
to nested colourings from those coming from other colourings.

\begin{theorem}\label{thm:Q-Borel}
    Let $\Delta$ be a simplicial complex, and let $\mC$ be a proper vertex $k$-colouring of $\Delta$.
    Set $R = K[x_1, \ldots, x_k, y_1, \ldots, y_k]$.  If $Q_k$ is the poset $\{x_1, \ldots, x_k, y_1, \ldots, y_k\}$
    with relations $x_i < y_i$ for each $1 \leq i \leq k$, then $I(\Delta, \mC)$ in $R$ is $Q_k$-Borel
    if and only if $\mC$ is a nested colouring with the nesting order.
\end{theorem}
\begin{proof}
    The colouring $\mC$ is nested if and only if for every $\sigma \in \Delta$, if $e_i(\sigma) > 0$, then there is a
    $\tau \in \Delta$ such that $e(\tau) = (e_1(\sigma), \ldots, e_i(\sigma) - 1, \ldots, e_k(\sigma))$, by
    Proposition~\ref{pro:alt-nested}.  That is, for every $m_\sigma \in I(\Delta, \mC)$, if $y_i$ divides $m_{\sigma}$,
    then $x_i m_{\sigma} / y_i \in I(\Delta, \mC)$.  The latter is precisely the $Q_k$-Borel property.
\end{proof}

From Proposition~\ref{pro:ufi-product}, we see that the product of two uniform face ideals is again a uniform face ideal.
Using the preceding classification of the uniform face ideals with respect to nested colourings, we see that
the resultant uniform face ideal comes from a nested colouring if and only if the factors come from nested colourings.

\begin{corollary}\label{cor:nested-ufi-products}
    Let $\Delta$ and $\Gamma$ be simplicial complexes, and let $\mC = C_1 \ddd C_k$ and $\mD = D_1 \ddd D_k$
    be nesting $k$-colourings of $\Delta$ and $\Gamma$ endowed with nesting orders, respectively.  If $\Sigma$ is the
    simplicial complex with proper vertex $k$-colouring $\mE$ satisfying $I(\Delta, \mC) \cdot I(\Gamma, \mD) = I(\Sigma, \mE)$,
    then $\mE$ is a nested colouring with the nesting order.
\end{corollary}
\begin{proof}
    By Theorem~\ref{thm:Q-Borel}, $I(\Delta, \mC)$ and $I(\Gamma, \mD)$ are $Q_k$-Borel.  Let $m_\rho$ be a minimal generator
    of $I(\Sigma, \mE)$, and suppose $m_\rho = m_\sigma \cdot m_\tau$, where $\sigma \in \Delta$ and $\tau \in \Gamma$.
    If $y_i$ divides $m_\rho$, then, without loss of generality, $y_i$ divides $m_\sigma$.  Since $I(\Delta, \mC)$
    is $Q_k$-Borel, $x_i m_\sigma / y_i$ is in $I(\Delta, \mC)$.  Hence $(x_i m_\sigma / y_i) \cdot m_\tau = x_i m_\rho / y_i$
    is in $I(\Sigma, \mE)$.  Thus $I(\Sigma, \mE)$ is $Q_k$-Borel, and so $\mE$ is a nested colouring with the nesting order
    by Theorem~\ref{thm:Q-Borel}.
\end{proof}

\subsection{Matroidal, polymatroidal, and weakly polymatroidal}\label{sub:matroidal}~

Let $R = K[x_1, \ldots, x_n]$, where $K$ is a field.  A monomial ideal $I$ of $R$ is \emph{polymatroidal} if every
minimal generator of $I$ has the same degree and for every pair of minimal generators $m = x_1^{a_1} \cdots x_n^{a_n}$
and $m' = x_1^{b_1} \cdots x_n^{b_n}$ and for all $i$ such that $a_i > b_i$, there exists a $j$ such that $b_j > a_j$
and $x_j m / x_i$ is a minimal generator of $I$.  Further, $I$ is \emph{matroidal} if it is squarefree and polymatroidal.

Conca and Herzog~\cite{CH} showed that polymatroidal ideals have linear quotients and hence linear resolutions. Since
the product of polymatroidal ideals is again polymatroidal, all powers of polymatroidal ideals have linear quotients and
hence linear resolutions.

A uniform face ideal is polymatroidal if and only if it is principal $Q_k$-Borel.  In particular, this implies that
the colouring used is a nested colouring with the nesting order.

\begin{proposition}\label{pro:polymatroidal}
    Let $\Delta$ be a simplicial complex, and let $\mC$ be a proper vertex $k$-colouring of $\Delta$.  The ideal
    $I(\Delta, \mC)$ is polymatroidal if and only if $I(\Delta, \mC)$ is a principal $Q_k$-Borel ideal.

    Hence $I(\Delta, \mC)$ is matroidal if and only if $\Delta$ is a simplex.
\end{proposition}
\begin{proof}
    Notice that, by the construction of $I(\Delta, \mC)$, the sum of the exponents of $x_i$ and $y_i$ on $m$ is $\#C_i$,
    for every $i$ and every minimal generator $m$ of $I(\Delta, \mC)$.  Hence every exchange of the kind considered
    for the polymatroidal property must exchange $x_i$ for $y_i$ or $y_i$ for $x_i$.

    By~\cite[Proposition~2.9]{FMS}, if $I(\Delta, \mC)$ is a principal $Q_k$-Borel ideal, then $I(\Delta, \mC)$ is
    polymatroidal.
    Suppose $I(\Delta, \mC)$ is \emph{not} a principal $Q_k$-Borel ideal.  That is, $I(\Delta, \mC)$ has two distinct
    $Q_k$-Borel generators $m_\sigma$ and $m_\tau$.  Since a $Q_k$-Borel move exchanges $x_i$ for $y_i$, there exist
    indices $i$ and $j$ such that $e_i(\sigma) > e_i(\tau)$ and $e_j(\tau) > e_j(\sigma)$.  The latter implies the
    exponent on $y_j$ is larger in $m_\tau$ than in $m_\sigma$, and further that the exponent on $x_j$ is smaller in
    $m_\tau$ than in $m_\sigma$.  Thus the only possible valid exchange of the kind considered for the polymatroidal
    property is $y_j m_\sigma / x_j$.  However, there is a $Q_k$-Borel move on $y_j m_\sigma / x_j$ that generates
    $m_\sigma$, thus $y_j m_\sigma / x_j$ is not in $I(\Delta, \mC)$.  Therefore, $I(\Delta, \mC)$ is not polymatroidal.

    The second claim follows as $I(\Delta, \mC)$ is squarefree precisely when the colour classes of $\mC$ have cardinality
    at most $1$.
\end{proof}

\begin{remark}\label{rem:polymatroidal}
    By Theorem~\ref{thm:Q-Borel}, we know that $I(\Delta, \mC)$ is $Q_k$-Borel precisely when $\mC$ is a nested
    $k$-colouring.  We note that $I(\Delta, \mC)$ is principal $Q_k$-Borel precisely when $\Delta$ is a clique
    complex of a complete $k$-partite graph and $\mC$ is the optimal colouring.
\end{remark}

As with stability, being polymatroidal is a very strong condition, especially with regard to uniform face ideals.  A more
useful property is being weakly polymatroidal.  Recall that a monomial ideal $I$ is \emph{weakly polymatroidal} if for every
pair of minimal generators $m = x_1^{a_1} \cdots x_n^{a_n}$ and $m' = x_1^{b_1} \cdots x_n^{b_n}$ such that
$a_1 = b_1, \ldots, a_{i-1} = b_{i-1}$, and $a_i > b_i$ for some $i$, there exists a $j > i$ such that $x_i m' / x_j \in I$.
Notice that the weakly polymatroidal property depends on the order of the variables of $R$.

Kokubo and Hibi~\cite{KH} proved that weakly polymatroidal ideals have linear quotients and hence linear resolutions.
However, unlike polymatroidal ideals, the product of weakly polymatroidal ideals need not be weakly polymatroidal.

A uniform face ideal is weakly polymatroidal precisely when the colouring used is a nested colouring.  Thus a uniform
face ideal with respect to a nested colouring with the nesting order has linear quotients and hence a linear resolution;
see Section~\ref{sec:resolution} for more detailed results thereon.

\begin{theorem}\label{thm:weakly-polymatroidal}
    Let $\Delta$ be a simplicial complex, and let $\mC$ be a proper vertex $k$-colouring of $\Delta$.
    If $R = K[x_1, \ldots, x_k, y_1, \ldots, y_k]$, then the ideal $I(\Delta, \mC)$ is weakly polymatroidal
    if and only if $\mC$ is a nested colouring with the nesting order.
\end{theorem}
\begin{proof}
    By the construction of $I(\Delta, \mC)$, the sum of the exponents of $x_i$ and $y_i$ on $m$ is $\#C_i$, for every
    $i$ and every minimal generator $m$ of $I(\Delta, \mC)$.  Hence every exchange of the kind considered for the weakly
    polymatroidal property must exchange $x_i$ for $y_i$.  Further still, if two minimal generators have the same
    exponents for $x_1, \ldots, x_n$, then they are the same minimal generators.

    Using Theorem~\ref{thm:Q-Borel}, we may instead show that $I(\Delta, \mC)$ is weakly polymatroidal if and only
    if $I(\Delta, \mC)$ is $Q_k$-Borel.

    Suppose first that $I(\Delta, \mC)$ is $Q_k$-Borel.  Let $m_\sigma$ and $m_\tau$ be a pair of minimal generators such
    that the first exponent that is different is at $x_i$, and suppose the exponent on $x_i$ in $m_\sigma$ is larger
    than the exponent on $x_i$ in $m_\tau$.  This implies that the exponent on $y_i$ in $m_\tau$ is positive, and so
    $x_i m_\tau / y_i \in I(\Delta, \mC)$ as $I(\Delta, \mC)$ is $Q_k$-Borel.  Thus $I(\Delta, \mC)$ is weakly polymatroidal.

    Now suppose $I(\Delta, \mC)$ is not $Q_k$-Borel.  Hence there exists a minimal generator $m_\tau$ of $I(\Delta, \mC)$
    such that $y_i$ divides $m_\tau$ but $x_i m_\tau / y_i$ is not in $I(\Delta, \mC)$.  Let $\sigma$ be $\tau \setminus C_i$.
    Thus $m_\sigma$ is a minimal generator of $I(\Delta, \mC)$ such that the first exponent that is different from that
    of $m_\tau$ is at $x_i$, and further $m_\sigma$ has a larger exponent on $x_i$ than $m_\tau$.  However, by assumption,
    $x_i m_\tau / y_i$ is not in $I(\Delta, \mC)$.  Therefore, $I(\Delta, \mC)$ is not weakly polymatroidal.
\end{proof}

We note that~\cite{KH} implies that $I(\Delta, \mC)$ has linear quotients with respect to the degree lexicographic order.  However,
we only focus on the presence of a linear resolution.

\section{The first syzygies}\label{sec:syzygies}

In this section, we describe the first syzygies of the uniform face ideals.  In particular, we find the first
$\ZZ$-graded Betti numbers of the uniform face ideals for some cases.  Further, we classify the uniform face ideals
that have a linear resolution.

To achieve this goal, we define a poset on the index vectors of a simplicial complex with respect to an
ordered colouring; see Figure~\ref{fig:Delta-P} for an example.

\begin{definition}\label{def:index-poset}
    Let $\Delta$ be a simplicial complex, and let $\mC$ be an ordered $k$-colouring of $\Delta$.  The
    \emph{index vector poset of $\Delta$ with respect to $\mC$} is the set $P(\Delta, \mC)$ of index
    vectors of faces of $\Delta$ with respect to $\mC$ partially ordered componentwise.
\end{definition}

\begin{figure}[!ht]
    \centering
    \includegraphics[scale=1.5]{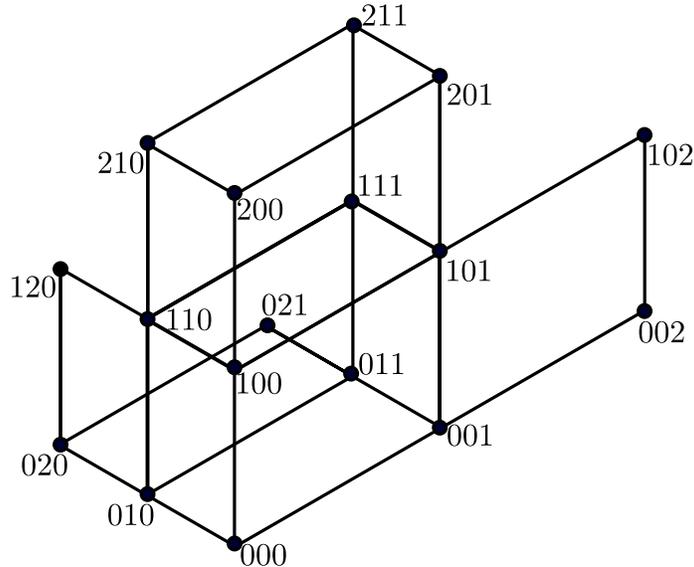}
    \caption{The Hasse diagram of the index vector poset $P(\Delta, \mC)$, where $\Delta = \langle abc, bcd, ce, de, df \rangle$ and
        $\mC = \{d,a\} \dcup \{b,e\} \dcup \{c,f\}$, as in Example~\ref{exa:nested}.}
    \label{fig:Delta-P}
\end{figure}

Given a simplicial complex $\Delta$, the \emph{face poset of $\Delta$} is the poset of faces of $\Delta$
partially ordered by inclusion.  It is easy to see that the face poset of $\Delta$ is isomorphic to the
index vector poset of $\Delta$ with respect to the singleton colouring $\mS = \{1\} \ddd \{n\}$.

Clearly, the index vector posets are finite subposets of $\NN_0^k$ partially ordered componentwise.
An index vector poset is an order ideal of $\NN_0^k$ if and only if the ordered colouring is nested.

\begin{proposition}\label{pro:nested-order-ideal}
    Let $\Delta$ be a simplicial complex, and let $\mC$ be an ordered $k$-colouring of $\Delta$.
    The poset $P(\Delta, \mC)$ is a finite order ideal of $\NN_0^k$ if and only if $\mC$ is a nested colouring
    with the nesting order.

    Moreover, every finite order ideal of $\NN_0^k$ arises this way.
\end{proposition}
\begin{proof}
    Suppose $P(\Delta, \mC)$ is a finite order ideal of $\NN_0^k$.  If $e = (e_1, \ldots, e_k) \in P(\Delta, \mC)$,
    then $e' = (e_1, \ldots, e_i - 1, \ldots, e_k) \in P(\Delta, \mC)$ for $i$ such that $e_i > 0$.  That is, if the
    face associated to $e$ is in $\Delta$, then so is the face associated to $e'$.  Notice that the face associated to $e'$
    is the face associated to $e$ with the vertex in the colour class $C_i$ replaced by the next lower vertex in the ordering
    (or removed altogether if $e_i = 1$).  This is precisely the condition for $\mC$ to be nested with the nesting order
    by Proposition~\ref{pro:alt-nested}.  The converse follows similarly.
\end{proof}

Let $R = K[x_1,\ldots,x_n]$ be the $n$-variate polynomial ring over a field $K$, and let $I$ be an ideal of $R$
minimally generated by the homogeneous forms $g_1, \ldots, g_t$.  Set $\varphi$ to be the map
\[
    \bigoplus_{i=1}^{t} R(-\deg{g_i}) \xrightarrow{[g_1, \ldots, g_t]} R.
\]
The \emph{(first) syzygies of $R/I$} are the generators of the kernel of $\varphi$.

The syzygies associated to the covering relations of $P(\Delta, \mC)$ are all minimal generators of the first syzygies
of $I(\Delta, \mC)$.

\begin{lemma}\label{lem:covering-syzygies}
    Let $\Delta$ be a simplicial complex, and let $\mC$ be an ordered $k$-colouring of $\Delta$.
    Suppose $\eps_\sigma$ represents the basis element that maps onto $m_\sigma$.
    If $e(\sigma) < e(\tau)$ is a covering relation of $P(\Delta, \mC)$, then the binomial
    \[
        \eps_\sigma \prod_{i=1}^{k} y_i^{e_i(\tau) - e_i(\sigma)} - \eps_\tau \prod_{i=1}^{k} x_i^{e_i(\tau) - e_i(\sigma)}
    \]
    is a minimal generator of the first syzygies of $I(\Delta, \mC)$.
\end{lemma}
\begin{proof}
    By~\cite[Lemma~15.1]{Ei}, or more generally Schreyer's Algorithm (see, e.g., \cite[Theorem~15.10]{Ei}), the first
    syzygies of $I(\Delta, \mC)$ are generated by the binomials
    \begin{equation*}
        \begin{split}
            s_{\sigma, \tau} = & \eps_\sigma \frac{m_\tau}{\gcd(m_\sigma, m_\tau)} - \eps_\tau \frac{m_\sigma}{\gcd(m_\sigma, m_\tau)} \\
                             = & \eps_\sigma \prod_{i=1}^{k} x_i^{\max(0, e_i(\sigma) - e_i(\tau))} y_i^{\max(0, e_i(\tau) - e_i(\sigma))}
                               - \eps_\tau \prod_{i=1}^{k} x_i^{\max(0, e_i(\tau) - e_i(\sigma))} y_i^{\max(0, e_i(\sigma) - e_i(\tau))},
        \end{split}
    \end{equation*}
    where $\sigma$ and $\tau$ are distinct faces of $\Delta$.  Clearly, $s_{\sigma, \tau} = -s_{\tau, \sigma}$.

    Let $e(\sigma) < e(\tau)$ be any covering relation of $P(\Delta, \mC)$.  Thus $s_{\sigma, \tau}$ simplifies to
    \[
        \eps_\sigma \prod_{i=1}^{k} y_i^{e_i(\tau) - e_i(\sigma)} - \eps_\tau \prod_{i=1}^{k} x_i^{e_i(\tau) - e_i(\sigma)}.
    \]
    Suppose $s_{\sigma, \tau} = a_1 s_{\gamma_1, \delta_1} + \cdots + a_t s_{\gamma_t, \delta_t}$, where $\gamma_i, \delta_i$
    are faces of $\Delta$.  Without loss of generality, we may assume $\gamma_1 = \sigma$ and $\delta_1 \neq \sigma$.
    As the coefficient on $\eps_\sigma$ in $s_{\sigma, \tau}$ is $a_1$ times the coefficient on $\eps_\sigma$ in $s_{\sigma, \delta_1}$,
    we have that
    \[
        \prod_{i=1}^{k} y_i^{e_i(\tau) - e_i(\sigma)} = a_1 \cdot \prod_{i=1}^{k} y_i^{e_i(\delta_1) - e_i(\sigma)}.
    \]
    This implies that $e(\tau) \geq e(\delta_1)$.  Since $e(\sigma) < e(\tau)$ is a covering relation, we must have that
    $\delta_1 = \tau$. Thus $s_{\sigma, \tau}$ is a minimal generator of the first syzygies of $I(\Delta, \mC)$.
\end{proof}

We thus see that the syzygies in the preceding lemma generate all of the syzygies if $P(\Delta, \mC)$ is a meet-semilattice.

\begin{corollary}\label{cor:meet-semilattice-syzygies}
    Let $\Delta$ be a simplicial complex, and let $\mC$ be an ordered $k$-colouring of $\Delta$.
    Suppose $\eps_\sigma$ represents the basis element that maps onto $m_\sigma$, and let $s_{\sigma, \tau}$ be the syzygy associated
    to $\eps_\sigma$ and $\eps_\tau$.  If $P(\Delta, \mC)$ is a meet-semilattice, then the syzygies $s_{\sigma, \tau}$ associated to
    covering relations $e(\sigma) < e(\tau)$  of $P(\Delta, \mC)$ minimally generate the first syzygies of $I(\Delta, \mC)$.
\end{corollary}
\begin{proof}
    See the first paragraph of the proof of Lemma~\ref{lem:covering-syzygies} for the definition of the syzygies $s_{\sigma, \tau}$.

    Suppose $\sigma$ and $\tau$ are distinct faces of $\Delta$, and suppose $\kappa_1, \ldots, \kappa_{t+1}$ is a sequence of faces
    of $\Delta$ such that $\kappa_1 = \sigma$, $\kappa_{t+1} = \tau$, and $\kappa_{i} < \kappa_{i+1}$ is a covering relation for
    $1 \leq i \leq t$.  We will show that $s_{\sigma, \tau}$ can be generated from the covering relation syzygies
    $s_{\kappa_{i}, \kappa_{i+1}}$, where $1 \leq i \leq t$.  We proceed by induction on $t$, the length of some chain of covering
    relations between $\sigma$ and $\tau$.  If $t = 1$, then we are done.  Suppose $t > 1$.  By induction, $s_{\kappa_{1}, \kappa_{t}}$
    can be generated from the covering relation syzygies $s_{\kappa_i, \kappa_{i+1}}$, where $1 \leq i \leq t-1$.  Since
    \begin{equation*}
        \begin{split}
              & \prod_{i=1}^{k} y_i^{e_i({\kappa_{t+1}}) - e_i({\kappa_t})} s_{\kappa_1,\kappa_t} +
                  \prod_{i=1}^{k} x_i^{e_i({\kappa_t}) - e_i({\kappa_1})} s_{\kappa_t, \kappa_{t+1}} \\
            = &  \prod_{i=1}^{k} y_i^{e_i({\kappa_{t+1}}) - e_i({\kappa_t})}
                 (\eps_{\kappa_1} \prod_{i=1}^{k} y_i^{e_i({\kappa_t}) - e_i({\kappa_1})}
                     - \eps_{\kappa_t} \prod_{i=1}^{k} x_i^{e_i({\kappa_t}) - e_i({\kappa_1})}) \\
              & + \prod_{i=1}^{k} x_i^{e_i({\kappa_t}) - e_i({\kappa_1})}
                (\eps_{\kappa_t} \prod_{i=1}^{k} y_i^{e_i({\kappa_{t+1}}) - e_i({\kappa_t})}
                       - \eps_{\kappa_{t+1}} \prod_{i=1}^{k} x_i^{e_i({\kappa_{t+1}}) - e_i({\kappa_t})}) \\
            = & \eps_{\kappa_1} \prod_{i=1}^{k} y_i^{(e_i({\kappa_{t+1}}) - e_i({\kappa_t})) + (e_i({\kappa_t}) - e_i({\kappa_1}))}
              - \eps_{\kappa_{t+1}} \prod_{i=1}^{k} x_i^{(e_i({\kappa_{t}}) - e_i({\kappa_1})) + (e_i({\kappa_{t+1}}) - e_i({\kappa_t}))} \\
            = & \eps_{\kappa_1} \prod_{i=1}^{k} y_i^{e_i({\kappa_{t+1}}) - e_i({\kappa_1})}
              - \eps_{\kappa_{t+1}} \prod_{i=1}^{k} x_i^{e_i({\kappa_{t+1}}) - e_i({\kappa_1})} \\
            = & s_{\kappa_1, \kappa_{t+1}},
        \end{split}
    \end{equation*}
    we have that $s_{\kappa_1, \kappa_{t+1}} = s_{\sigma, \tau}$ can be generated from $s_{\kappa_{1}, \kappa_{t}}$ and
    $s_{\kappa_{t}, \kappa_{t+1}}$, the former of which can be generated from the covering relation syzygies
    $s_{\kappa_i, \kappa_{i+1}}$, where $1 \leq i \leq t-1$.

    Now suppose $\sigma$ and $\tau$ are distinct faces of $\Delta$ such that $e(\sigma)$ and $e(\tau)$ are incomparable, and
    have a meet, say, $\mu$.  If we proceed as in the previous displayed equation, then we get that $s_{\sigma, \tau}$ can be
    generated from $s_{\mu, \sigma}$ and $s_{\mu, \tau}$.  By the previous step, we see that $s_{\mu, \sigma}$ and $s_{\mu, \tau}$
    can be generated by covering relation syzygies, and thus so can $s_{\sigma, \tau}$.
\end{proof}

\begin{example}\label{exa:non-meet-semilattice}
    Let $\Delta = \langle ab, cd \rangle$, and let $\mC = \{c,a\} \dcup \{b, d\}$; see Figure~\ref{fig:non-meet-semilattice}(i).
    Clearly, $\mC$ is not a nested colouring of $\Delta$; indeed, the only nested colouring of $\Delta$ is the singleton colouring.
    Suppose $\eps_\sigma$ represents the basis element that maps onto $m_\sigma$, and let $s_{\sigma, \tau}$ be the syzygy
    associated to $\eps_\sigma$ and $\eps_\tau$.

    \begin{figure}[!ht]
        \begin{minipage}[b]{0.48\linewidth}
            \centering
            \includegraphics[scale=1.75]{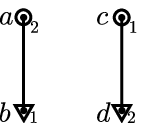}\\
            \emph{(i) The complex $\Delta$ with colouring $\mC$.}
        \end{minipage}
        \begin{minipage}[b]{0.48\linewidth}
            \centering
            \includegraphics[scale=1.5]{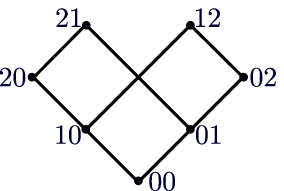}\\
            \emph{(ii) The index vector poset $P(\Delta, \mC)$.}
        \end{minipage}
        \caption{An index vector poset that is not a meet-semilattice.}
        \label{fig:non-meet-semilattice}
    \end{figure}

    By Lemma~\ref{lem:covering-syzygies}, $s_{b,ab}$ and $s_{c,cd}$ are both minimal syzygies of $R/I(\Delta, \mC)$,
    since $e(b) < e(ab)$ and $e(c) < e(cd)$ are both covering relations in $P(\Delta, \mC)$; see Figure~\ref{fig:non-meet-semilattice}(ii).
    However, both syzygies are quadratic.  Moreover, the syzygy $s_{ab,cd}$ is a quadratic minimal syzygy.
\end{example}

From the preceding lemma we can provide a lower bound on the first $\ZZ$-graded Betti numbers of a uniform face ideal.  Moreover,
equality holds when the index vector poset is a meet-semilattice, by the preceding corollary.

\begin{corollary}\label{cor:first-betti}
    Let $\Delta$ be a simplicial complex on $[n]$, and let $\mC$ be an ordered $k$-colouring of $\Delta$.
    Set $A$ to be the set $\{(e(\sigma), e(\tau)) \st e(\sigma) < e(\tau) \mbox{~is a covering relation}\}$.
    If $1 \leq i \leq n$, then
    \[
        \beta_{1,n+i}(I(\Delta, \mC)) \geq \#\left\{ (e(\sigma), e(\tau)) \in A \st \sum_{j=1}^k (e_j(\tau) - e_j(\sigma)) = i\right\};
    \]
    equality holds if $P(\Delta, \mC)$ is a meet-semilattice.

    In particular, the first total Betti number of $I(\Delta, \mC)$ satisfies
    \[
        \beta_1(I(\Delta, \mC)) \geq \sum_{j=1}^{1 + \dim \Delta} j f_{j-1}(\Delta).
    \]
\end{corollary}
\begin{proof}
    The first bound follows immediately from Lemma~\ref{lem:covering-syzygies}.  Equality in the bound follows from
    Corollary~\ref{cor:meet-semilattice-syzygies}.

    To see the second bound, notice that, for each face $\sigma$ in $\Delta$, $e(\sigma)$ covers at least $\#\sigma$ different
    index vector elements, one for each nonzero entry in $e(\sigma)$.
\end{proof}

We note that the second bound need not be an equality, even for meet-semilattices.

\begin{example}\label{exa:awkward-covering-relation}
    Let $\Delta = \langle ab, cd \rangle$ be the complex in Example~\ref{exa:non-meet-semilattice}.  The $f$-vector
    of $\Delta$ is $f(\Delta) = (1,4,2)$, and hence $\beta_1(I(\Delta, \mD)) \geq 8$.

    Consider now the (non-nested) colouring $\mD = \{a,c\} \dcup \{b, d\}$; see Figure~\ref{fig:awkward}(i).
    In this case, the index vector poset $P(\Delta, \mD)$ is a meet-semilattice; see Figure~\ref{fig:awkward}(ii).
    However, $e(ab) < e(cd)$ is a covering relation, increasing the first Betti number of $I(\Delta, \mD)$.  In particular,
    $\beta_1(I(\Delta,\mD)) = 9$.

    \begin{figure}[!ht]
        \begin{minipage}[b]{0.48\linewidth}
            \centering
            \includegraphics[scale=1.75]{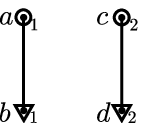}\\
            \emph{(i) The complex $\Delta$ with colouring $\mD$.}
        \end{minipage}
        \begin{minipage}[b]{0.48\linewidth}
            \centering
            \includegraphics[scale=1.5]{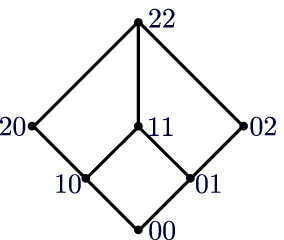}\\
            \emph{(ii) The index vector poset $P(\Delta, \mD)$.}
        \end{minipage}
        \caption{An index vector poset that is a meet-semilattice, but has ``quadratic'' covering relations.}
        \label{fig:awkward}
    \end{figure}
\end{example}

By the preceding corollary, a uniform face ideal has a linear resolution if and only if the colouring used is nested.

\begin{theorem}\label{thm:ufi-linear-nested}
    Let $\Delta$ be a simplicial complex, and let $\mC$ be an ordered $k$-colouring of $\Delta$.
    The following are equivalent:
    \begin{enumerate}
        \item $I(\Delta, \mC)$ has linear first syzygies,
        \item $I(\Delta, \mC)$ has a linear resolution,
        \item every power of $I(\Delta, \mC)$ has a linear resolution, and
        \item $\mC$ is nested and endowed with the nesting order.
    \end{enumerate}
\end{theorem}
\begin{proof}
    Clearly, claim (iii) implies claim (ii), and claim (ii) implies claim (i).

    By Theorem~\ref{thm:weakly-polymatroidal} and Corollary~\ref{cor:nested-ufi-products}, if $\mC$ is nested and endowed with
    the nesting order, then every power of $I(\Delta, \mC)$ is weakly polymatroidal.  Thus, by \cite[Corollary~1.5]{KH},
    every power of $I(\Delta, \mC)$ has a linear resolution.  That is, claim (iv) implies claim (iii).

    Suppose claim (i) holds, i.e., $I(\Delta, \mC)$ has linear first syzygies.  Thus Corollary~\ref{cor:first-betti}
    implies that every covering relation of $P(\Delta, \mC)$ is of the form $e(\sigma) > e(\tau)$, where
    $e(\tau) = (e_1(\sigma), \ldots, e_i(\sigma) - 1, \ldots, e_k(\sigma))$ for some $i$.  That is, $P(\Delta, \mC)$ is
    an order ideal of $\NN_0^k$.  By Proposition~\ref{pro:nested-order-ideal}, this is equivalent to claim (iv).
\end{proof}

\section{Cellular resolutions}\label{sec:resolution}

In this section, we describe a minimal linear cellular resolution of $I(\Delta, \mC)$, when $\mC$ is nested.
This resolution is supported on a collapsible cubical complex described by the poset $P(\Delta, \mC)$.

\subsection{A collapsible cubical complex}\label{sub:cubical}~

We first see that $P(\Delta, \mC)$ is a meet-distributive meet-semilattice precisely when $\mC$ is nested.

\begin{proposition}\label{pro:meet}
    Let $\Delta$ be a simplicial complex, and let $\mC$ be an ordered $k$-colouring of $\Delta$.
    The colouring $\mC$ is nested with the nesting order if and only if $P(\Delta, \mC)$ is a finite
    meet-distributive meet-semilattice.
\end{proposition}
\begin{proof}
    By Proposition~\ref{pro:nested-order-ideal}, we may show instead that $P(\Delta, \mC)$ is a finite order
    ideal of $\NN_0^k$ if and only if $P(\Delta, \mC)$ is a finite meet-distributive meet-semilattice.

    Clearly, $\NN_0^k$ is a meet-distributive meet-semilattice.  Hence $P(\Delta, \mC)$ is a finite meet-distributive
    meet-semilattice if $P(\Delta, \mC)$ is a finite order ideal of $\NN_0^k$.

    Now suppose $P = P(\Delta, \mC)$ is a finite meet-distributive meet-semilattice.  If $k = 1$, then
    $P(\Delta, \mC)$ is clearly a finite order ideal of $\NN_0^k$.  Suppose $k \geq 2$.  We proceed by induction
    on the number of vertices of $P$.  If $P$ has one vertex, then it must be $(0,\ldots,0)$; hence $P$ is indeed a
    finite order ideal of $\NN_0^k$.  Suppose $P$ has $N$ vertices, and let $e(\tau)$ be a maximal element of $P$.
    Hence $\tau$ is a facet of $\Delta$, and so $P' = P \setminus \{e(\tau)\} = P(\Delta \setminus \{\tau\}, \mC)$.
    Since $P'$ has $N-1$ vertices and is a finite order ideal of $P$, $P'$ is a meet-distributive
    meet-semilattice and hence a finite order ideal of $\NN_0^k$ by induction.

    Let $e(\sigma_1), \ldots, e(\sigma_t)$ be the elements of $P$ that are covered by $e(\tau)$, and let
    $e(\mu)$ be the meet in $P$.  Since $P$ is meet-distributive, the interval $[e(\mu), e(\tau)]$ is
    isomorphic to $B_t$, and thus $[e(\mu), e(\sigma_i)]$ is isomorphic to $B_{t-1}$.  Moreover, each interval
    $[e(\mu), e(\sigma_i)]$ is in $P'$, and so $e(\mu)$ and $e(\sigma_i)$ differ by $1$ in exactly $t-1$
    entries.  As $[e(\mu), e(\tau)]$ is isomorphic to $B_t$, we further see that precisely $t-2$ of these
    positions are common between $e(\sigma_i)$ and $e(\sigma_j)$, where $i \neq j$.  Thus $e(\mu)$ and
    $e(\tau)$ differ in precisely $t$ positions, i.e., $e(\sigma_i)$ and $e(\tau)$ differ in precisely $1$ position.
    Without loss of generality, suppose $e_i(\tau) - e_i(\sigma_i) \geq 1$ for each $i$.

    Assume $e_j(\tau) - e_j(\sigma_j) \geq 2$ for some $j$; without loss of generality, assume $j = 1$.
    Let $\tau' = \tau \setminus C_2$.  Hence $\tau'$ is a face of $\Delta$ and $e_i(\tau') = e_i(\tau)$ if
    $i \neq 2$ and $e_2(\tau') = 0$.  Notice that $e(\tau')$ is incomparable to each $e(\sigma_i)$, except
    possibly $e(\sigma_2)$.  If $e(\tau')$ is comparable to $e(\sigma_2)$, then $e(\tau') > e(\sigma_2)$,
    which contradicts the assumption that  $e(\sigma_2) < e(\tau)$ is a covering relation.  If $e(\tau')$ is not comparable
    to $e(\sigma_2)$, then there must exist another element covered by $e(\tau)$ that is not one of the
    $e(\sigma_i)$, contradicting our choice of the $e(\sigma_i)$.  Regardless of the comparability
    of $e(\tau')$ and $e(\sigma_2)$, we have that $e_j(\tau) - e_j(\sigma_j) = 1$ for all $j$.
    Thus $P$ is a finite order ideal of $\NN_0^k$.
\end{proof}

A \emph{cubical complex} $C$ on $[n]$ is a collection of subsets of $[n]$ partially ordered by inclusion such
that the following properties hold:
\begin{enumerate}
    \item $\emptyset \in C$.
    \item $\{i\} \in C$ for $i \in [n]$.
    \item For every nontrivial $F \in C$, the interval $[\emptyset, F] = \{G \in C \st G \subset F\}$ is
        isomorphic to a boolean poset.
    \item If $F, G \in C$, then $F \cap G \in C$.
\end{enumerate}
The elements of $C$ are \emph{faces}, and the maximal faces are \emph{facets}.  The \emph{dimension of a face $F$} is
$\dim{F} := i$, such that $[\emptyset, F] \cong B_i$, and the \emph{dimension of $C$} is the maximum dimension $\dim{C}$
of its faces.  The \emph{$f$-vector} (or \emph{face vector}) \emph{of $C$} is the $(d+1)$-tuple
$f(C) = (f_{-1}, \ldots, f_{d-1})$, where $f_i$ is the number of faces of dimension $i$ in $C$ and $d = \dim{C} + 1$.

We can define a cubical complex by ``filling in'' the boolean intervals of $P(\Delta, \mC)$ with cubes.

\begin{proposition}\label{pro:cubical}
    Let $\Delta$ be a simplicial complex, and let $\mC$ be a nested $k$-colouring of $\Delta$ endowed with the nesting
    order. If $C$ is $\emptyset$ together with the singletons of elements of $P(\Delta, \mC)$ and with the collection
    of intervals $[e(\sigma), e(\tau)]$ of $P(\Delta, \mC)$ isomorphic to boolean posets, then $C$ is a cubical complex.
\end{proposition}
\begin{proof}
    By construction, $C$ contains $\emptyset$ and the singletons of elements of $P(\Delta, \mC)$.  Moreover,
    we have that $P(\Delta, \mC)$ is a meet-distributive meet-semilattice by Proposition~\ref{pro:meet}.

    For any two faces $\sigma$ and $\tau$ of $\Delta$, if $[e(\sigma), e(\tau)]$ of $P(\Delta, \mC)$ is isomorphic to $B_d$,
    then $e(\sigma)$ and $e(\tau)$ differ in precisely $d$ indices $i_1, \ldots, i_d$, and for each such index $i$,
    $e_i(\tau) = e_i(\sigma) + 1$.  Let $\eps(j)$ be the $k$-tuple such that $\eps(j)_i = 0$ if $i \neq j$ and $\eps(j)_j = 1$.
    We have that $[e(\sigma), e(\tau)]$ consists of the vectors of the form $e(\sigma) + \sum_{j \in J} \eps(j)$, where
    $J$ is a subset of $\{i_1, \ldots, i_d\}$.  In particular, every interval of $[e(\sigma), e(\tau)]$ is isomorphic
    to a boolean poset, and is further an interval of $P(\Delta, \mC)$.  That is, every interval of $[e(\sigma), e(\tau)]$
    is also in $C$.  Hence, for any nontrivial $F \in C$, the interval $[\emptyset, F]$ in $C$ is isomorphic to a boolean poset.

    Let $F$ and $G$ be elements of $C$, i.e., $F = [e(\gamma), e(\delta)]$ and $G = [e(\sigma), e(\tau)]$ are intervals
    of $P(\Delta, \mC)$ that are isomorphic to boolean posets.  Set $\ell$ to be the componentwise maximum of $e(\gamma)$
    and $e(\sigma)$, and set $u$ to be the componentwise minimum of $e(\delta)$ and $e(\tau)$.  Clearly, we have
    $F \cap G = [e(\gamma), e(\delta)] \cap [e(\sigma), e(\tau)] = [\ell, u]$, as intervals of $P(\Delta, \mC)$.  In particular,
    if $\ell$ or $u$ is not in $P(\Delta, \mC)$, then the intersection is empty.  If $F \cap G$ is not empty, then
    $\ell$ and $u$ are in $F = [e(\gamma), e(\delta)]$; hence $[\ell, u]$ is itself isomorphic to a boolean poset and
    thus is in $C$.
\end{proof}

We formalise this cubical complex.

\begin{definition}\label{def:cubical}
    Let $\Delta$ be a simplicial complex, and let $\mC$ be a nested $k$-colouring of $\Delta$
    endowed with the nesting order.  The \emph{cubical complex of $\Delta$ with respect to $\mC$} is
    the cubical complex $C(\Delta, \mC)$ with vertex set $P(\Delta, \mC)$ and with faces
    consisting of intervals $[e(\sigma), e(\tau)]$ of $P(\Delta, \mC)$ isomorphic to boolean posets.
\end{definition}

\begin{remark}\label{rem:cubical}
    As will be seen in Section~\ref{sub:ferrers}, $P(\Delta, \mC)$ is a $k$-uniform Ferrers hypergraph on the
    vertex set $\{0, \ldots, \#C_1\} \ddd \{0, \ldots, \#C_k\}$.  Thus $C(\Delta, \mC)$ can be seen as a
    specialisation of the \emph{complex-of-boxes inside $P(\Delta, \mC)$}, as defined by Nagel and Reiner~\cite[Definition~3.11]{NR},
    where we restrict to boxes that have $1$ or $2$ vertices from each vertex set $\{0, \ldots, \#C_i\}$.
\end{remark}

\begin{example}\label{exa:Delta-C}
    Let $\Delta = \langle abc, bcd, ce, de, df \rangle$ be a simplicial complex with nested colouring
    $\mC = \{d,a\} \dcup \{b,e\} \dcup \{c,f\}$, as given in Example~\ref{exa:nested}.  The index vector
    poset $P(\Delta, \mC)$ is illustrated in Figure~\ref{fig:Delta-P}.  Filling in the cubes of $P(\Delta, \mC)$
    and rotating the image, we get the cubical complex $C(\Delta, \mC)$ as illustrated in Figure~\ref{fig:Delta-C}.
    The labels in the figure are the faces of $\Delta$ that correspond to the monomials at the vertices (note that
    the faces corresponding to $\emptyset$ and $d$ are hidden).
    \begin{figure}[!ht]
        \centering
        \includegraphics[scale=1.5]{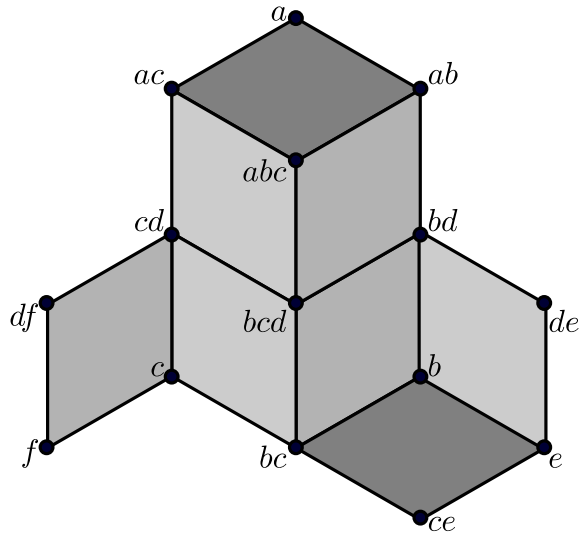}
        \caption{The cubical complex $C(\Delta, \mC)$, where $\Delta = \langle abc, bcd, ce, de, df \rangle$ and
            $\mC = \{d,a\} \dcup \{b,e\} \dcup \{c,f\}$, as in Example~\ref{exa:nested}.  The darkly-shaded faces
            are the tops.}
        \label{fig:Delta-C}
    \end{figure}
\end{example}

Let $X$ and $Y$ be cell complexes such that $Y \subset X$.  There is an \emph{elementary collapse of $X$ on $Y$}
if $X = Y \dcup F \dcup G$, where $F$ and $G$ are faces of $X$ such that $G$ is the unique face of $X$
properly containing $F$. The complex $X$ is \emph{collapsible} if there exists a sequence of elementary
collapses starting with $X$ and ending with a point.

\begin{lemma}\label{lem:collapsible}
    Let $\Delta$ be a simplicial complex, and let $\mC$ be a nested $k$-colouring of $\Delta$
    endowed with the nesting order.  The cubical complex $C(\Delta, \mC)$ is collapsible.
\end{lemma}
\begin{proof}
    We proceed by induction on the total number of faces of $\Delta$.  Suppose $\Delta$ has one face, i.e.,
    $\Delta = \langle \emptyset \rangle$.  In this case, $\mC = \emptyset$ and $P(\Delta, \mC)$ has precisely
    one vertex.  Hence $C(\Delta, \mC)$ is a single vertex and is trivially collapsible.

    Suppose $\Delta$ has more than one face, and let $\mC$ be a nested $k$-colouring of $\Delta$ endowed with
    the nesting order.  Pick any face $\tau$ of $\Delta$ such that $e(\tau)$ is a maximal element of
    $P(\Delta, \mC)$.  Clearly, $\mC$ is also a nested $k$-colouring of $\Delta \setminus \{\tau\}$, and the
    latter has one less face than $\Delta$.  Hence $C(\Delta \setminus \{\tau\}, \mC)$ is collapsible by induction.

    We will show that there is a sequence of elementary collapses of $C(\Delta, \mC)$ onto
    $C(\Delta \setminus \{\tau\}, \mC)$.  Let $\kappa_1, \ldots, \kappa_d$ be the elements of $P(\Delta, \mC)$
    covered by $e(\tau)$, and let $\mu_A$ be the meet of the elements $\{ \kappa_i \st i \in A\}$ in $P(\Delta, \mC)$,
    where $A$ is a nontrivial subset of $[d]$; the $\mu_A$ exist as $P(\Delta, \mC)$ is a meet-semilattice.
    Define $\mu_\emptyset = e(\tau)$.

    As $P(\Delta, \mC)$ is meet-distributive, the intervals $[\mu_A, e(\tau)]$ of $P(\Delta, \mC)$ are
    isomorphic to $B_{\#A}$, for each nontrivial subset $A$ of $[d]$; thus each $[\mu_A, e(\tau)]$ is a face
    of $C(\Delta, \mC)$.  Notice that the intervals $[\mu_A, e(\tau)]$ are precisely the faces of $C(\Delta, \mC)$
    containing $e(\tau)$.  Thus $C(\Delta \setminus \{\tau\}, \mC)$ is $C(\Delta, \mC)$ with the intervals
    $[\mu_A, e(\tau)]$ removed.

    If $X = Y \dcup F \dcup G$, where $F \subsetneq G$, is an elementary collapse, then we will say that
    $Y$ is a collapse of $F \subset G$ in $X$.  There are $d$ rounds of elementary collapses, and the $i^{\rm th}$
    round has $\binom{d-1}{i}$ elementary collapses in it.  The collapses in each round can be done in any order.
    In the first round, we collapse $C(\Delta, \mC)$ by $[\mu_{[d]}, e(\tau)] \subset [\mu_{[d-1]}, e(\tau)]$.
    In the $i^{\rm th}$ round, where $2 \leq i \leq d$, we sequentially collapse the complex by all pairs of intervals
    $[\mu_{B}, e(\tau)] \subset [\mu_{B \dcup \{d\}}, e(\tau)]$, where $B$ is a subset of $[d-1]$ of cardinality
    $d-i$.  By construction, $[\mu_A, e(\tau)] \subset [\mu_B, e(\tau)]$ if and only if $A \subset B \subset [d]$.
    Thus at the start of the $i^{\rm th}$ round all of the subsets of $[d]$ containing the $d-i$ subsets of $[d-1]$
    have been removed, and so each collapse is indeed an elementary collapse.

    We have thus described a sequence of elementary collapses of $C(\Delta, \mC)$ on $C(\Delta \setminus \{\tau\}, \mC)$.
    Therefore, $C(\Delta, \mC)$ is collapsible, as $C(\Delta \setminus \{\tau\}, \mC)$ is collapsible.
\end{proof}

\begin{remark}\label{rem:stanley}
    The construction of $C(\Delta, \mC)$ given in Definition~\ref{def:cubical} follows from a
    remark of Stanley~\cite[Topological Remark, pg.\ 417]{St}.  Further, the collapsibility of
    $C(\Delta, \mC)$ as presented in Lemma~\ref{lem:collapsible} is also derived from the remark.
\end{remark}

\subsection{A cellular resolution}\label{sub:resolution}~

The concept of a cellular resolution was introduced by Bayer and Sturmfels~\cite{BS} as a way to generate free
resolutions from a set of monomials via the combinatorial structure.  The presentation here follows
the presentation of Miller and Sturmfels~\cite[Chapter~4]{MS}.

We first generalise both simplicial and cubical complexes by using polyhedra.  A \emph{polyhedral cell complex}
$X$ is a finite collection of convex polytopes such that all faces of each $P \in X$ are in $X$ and $P \cap Q$
forms a face of both $P$ and $P$, for all $P, Q \in X$.  The elements of $X$ are \emph{faces}, and the maximal faces
are \emph{facets}.  The $0$-dimensional faces of $X$ are the \emph{vertices}.

Let $R = K[x_1, \ldots, x_n]$ be a polynomial ring.  A \emph{labeled cell complex} $X$ is a polyhedral cell complex
with each vertex labeled by a monomial in $R$ and each face $P \in X$ labeled by the least common multiple $x_P$
of the labels of the vertices contained in the face.  By convention, we label $\emptyset$ by $x_\emptyset = 1 \in R$.
From the structure of $X$, we obtain the \emph{cellular free complex $\mF(X)$ supported on $X$}, which is a
$\ZZ^n$-graded complex of free $R$-modules in which the following properties hold:
\begin{enumerate}
    \item For $i \geq -1$, the $i^{\rm th}$ term $\mF_i(X)$ is the free $R$-module with basis elements $\eps_P$,
        where $P$ is an $i$-dimensional face of $X$.  In particular, the empty face $\emptyset$ is the unique
        $-1$-dimensional face of $X$, and $\mF_{-1}(X)$ is a rank $1$ free $R$-module with basis element $\eps_\emptyset$
        of multidegree $\mathbf{0} \in \ZZ^n$.
    \item The differential $\delta$ of $\mF(X)$ is defined $R$-linearly by
        \[
            \delta(\eps_P) := \sum_{Q \mbox{~facet of~} P} \sgn(P, Q) \frac{x_P}{x_Q}\eps_Q,
        \]
        where $|\sgn(P,Q)| \in \{-1, 1\}$ is the incidence function described by boundary map of the chain complex of $X$.
\end{enumerate}

Set $I$ to be the monomial ideal generated by the monomial labels of the vertices of $X$.  In this case, we have
that $R/I$ is the cokernel of the map $\delta: \mF_0(X) \rightarrow \mF_{-1}(X)$.  The cellular free complex $\mF(X)$
supported on $X$ is a \emph{cellular resolution of $R/I$} if $\mF(X)$ is acyclic, that is, only has homology in degree $0$.
Moreover, $\mF(X)$ is a minimal free resolution of $R/I$ if and only if the labels $x_Q$ and $x_P$ differ for all faces $Q$
and $P$ of $X$ such that $Q$ is a proper face of $P$ (see~\cite[Remark~1.4]{BS}).

Bayer and Sturmfels classified precisely when a cellular free complex is a resolution by properties of the labeled
cell complex from which it arises.  For each multidegree $\mathbf{a} \in \ZZ^n$, let $X_{\leq \mathbf{a}}$ be the
subcomplex $\{P \st x_P \mbox{~divides~} x^{\mathbf{a}}\}$ of $X$.  Recall that a polyhedral cell complex is \emph{acyclic}
if it is either empty or has zero reduced homology.

\begin{proposition}{\cite[Proposition~1.2]{BS}}\label{pro:BS-resolution}
    Let $X$ be a labeled cell complex with labels in $R = K[x_1, \ldots, x_n]$, and let $I$ be the monomial ideal
    generated by the labels of the vertices of $X$.  The complex $\mF(X)$ is a free resolution of $R/I$ if and only
    if $X_{\leq \mathbf{a}}$ is acyclic over $K$ for all $\mathbf{a} \in \ZZ^n$.
\end{proposition}

Moreover, the $\ZZ^n$-graded Betti numbers of $R/I$ can be extracted from $X$ if it supports a cellular resolution.

\begin{theorem}{\cite[Theorem~1.11]{BS}}\label{thm:BS-betti}
    Let $X$ be a labeled cell complex with labels in $R = K[x_1, \ldots, x_n]$, and let $I$ be the monomial ideal
    generated by the labels of the vertices of $X$.  If $\mF(X)$ is a free resolution of $R/I$, then
    \[
        \beta_{i, \mathbf{a}}(R/I) = \dim \tilde{H}_{i-1}(X_{\leq \mathbf{a}}; K),
    \]
    where $\mathbf{a} \in \ZZ^n$ and $\tilde{H}_{-}$ denotes reduced homology.
\end{theorem}

The cubical complex $C(\Delta, \mC)$ supports a minimal cellular resolution of $R/I(\Delta, \mC)$, if $\mC$ is nested.

\begin{theorem}\label{thm:cellular-resolution}
    Let $\Delta$ be a simplicial complex on $[n]$, let $\mC$ be a nested $k$-colouring of $\Delta$
    endowed with the nesting order, and let $R = K[x_1, \ldots, x_k, y_1, \ldots, y_k]$, where
    $K$ is any field.  If $C(\Delta, \mC)$ is labeled so that $e(\sigma)$ has the monomial label $m_{\sigma}$,
    for $\sigma \in \Delta$, then $F(C(\Delta, \mC))$ is a minimal cellular resolution of $R/I(\Delta, \mC)$.
\end{theorem}
\begin{figure}[!ht]
    \centering
    \includegraphics[scale=1.5]{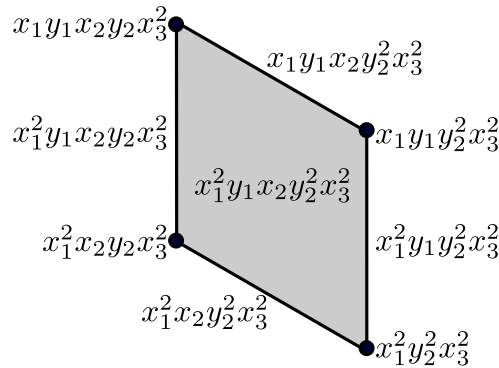}
    \caption{The monomial labeling of the face of $C(\Delta, \mC)$ in Figure~\ref{fig:Delta-C} that
        corresponds to the vertices $b, e, bd, de$.}
    \label{fig:Delta-C-face}
\end{figure}
\begin{proof}
    Clearly, $I(\Delta, \mC)$ is generated by the labels of the vertices of $C(\Delta, \mC)$.  By
    Lemma~\ref{lem:collapsible}, $C(\Delta, \mC)$ is collapsible.  Hence by Proposition~\ref{pro:BS-resolution}
    (i.e., \cite[Proposition~1.2]{BS}) $F(C(\Delta, \mC))$ is a free resolution of $R/I(\Delta, \mC)$.

    Every face of $C(\Delta, \mC)$ is an interval $[e(\sigma), e(\tau)]$ isomorphic to $B_d$, where $d$ is the
    number of vertices of $[e(\sigma), e(\tau)]$ covered by $e(\tau)$; hence $\sum_{i=1}^{k}(e_i(\tau) - e_i(\sigma)) = d$
    and $\max\{e_i(\tau) - e_i(\sigma) \st 1 \leq i \leq k\} = 1$.  Thus the least common multiple of the labels
    of the vertices in $[e(\sigma), e(\tau)]$ is $\lcm(m_\sigma, m_\tau)$, which has degree $n + d$.

    In particular, every face of $C(\Delta, \mC)$ strictly contained in another face must have a different label,
    as the degrees do not match.  Thus $F(C(\Delta, \mC))$ is a minimal resolution, by the remark preceding
    Proposition~\ref{pro:BS-resolution} (i.e., \cite[Proposition~1.2]{BS}).
\end{proof}

From this we can extract the $\ZZ$-graded Betti numbers of $I(\Delta, \mC)$, when $\mC$ is nested.

\begin{theorem}\label{thm:betti}
    Let $\Delta$ be a simplicial complex on $[n]$, let $\mC$ be a nested $k$-colouring of $\Delta$
    endowed with the nesting order, and let $R = K[x_1, \ldots, x_k, y_1, \ldots, y_k]$, where $K$ is
    any field.  If $0 \leq i \leq 1 + \dim \Delta$, then
    \[
        \beta_{i, n+i}(I(\Delta, \mC)) = \sum_{j=i}^{1 + \dim \Delta} \binom{j}{i} f_{j-1}(\Delta) = f_i(C(\Delta, \mC));
    \]
    otherwise, $\beta_{i, j}(I(\Delta, \mC)) = 0$.
\end{theorem}
\begin{proof}
    For $\mathbf{a} \in \ZZ^{2k}$, $\dim \tilde{H}_{i-1}(X_{\leq \mathbf{a}}; K)$ is $0$ if $x^{\mathbf{a}}$ is not
    a label of $C(\Delta, \mC)$ and $1$ otherwise.  By the construction of $C(\Delta, \mC)$, the faces of
    $C(\Delta, \mC)$ with a label of degree $n+i$ correspond precisely to the intervals of $P(\Delta, \mC)$
    isomorphic to $B_i$.  Thus using Theorem~\ref{thm:BS-betti} (i.e., \cite[Theorem~1.11]{BS}), we see that
    $\beta_{i, j}(I(\Delta, \mC))$  is the number of intervals of $P(\Delta, \mC)$ isomorphic to $B_i$ if
    $j = n + i$, and zero otherwise.

    For each $\tau \in \Delta$, $e(\tau)$ covers precisely $\dim{\tau} + 1$ elements of $P(\Delta, \mC)$.
    Hence $e(\tau)$ is the maximal element of $\binom{\dim{\tau}+1}{i}$ intervals isomorphic to $B_i$.
    Moreover, all intervals of $P(\Delta, \mC)$ isomorphic to $B_i$ are counted precisely once in this fashion.
    Hence there are $\sum_{j=i}^{1 + \dim \Delta} \binom{j}{i} f_{j-1}(\Delta)$ intervals isomorphic to $B_i$.
\end{proof}

\begin{remark}\label{rem:betti}
    The preceding corollary deserves several remarks.
    \begin{enumerate}
        \item The $\ZZ$-graded Betti numbers of $I(\Delta, \mC)$ depend only on $\Delta$, when $\mC$ is nested.
        \item The value for $\beta_{1,n+1}(I(\Delta, \mC))$ agrees with the bound given in Corollary~\ref{cor:first-betti}.
        \item Biermann and Van Tuyl~\cite[Theorem~13]{BVT} derived the $\ZZ$-graded Betti numbers of a large family
             of \emph{squarefree} monomial ideals.  This family includes $I(\Delta, \mS)$, where $\mS = \{1\} \ddd \{n\}$
             is the singleton colouring; see Section~\ref{sub:singleton} for more details.
        \item An exercise of Stanley~\cite[Exercise~3.47(b)]{St} provides the inspiration for counting the intervals of
            $P(\Delta, \mC)$ isomorphic to $B_i$.  Suppose $\mC$ is a nested $k$-colouring of $\Delta$.  If
            $b_i$ is the number of intervals of $P(\Delta, \mC)$ isomorphic to $B_i$, and $f_i$ is the number of
            elements of $P(\Delta, \mC)$ that cover exactly $i$ elements, then
            \[
                \sum_{i = 0}^k b_i x^i = \sum_{i = 0}^k f_{i-1}(1 + x)^i.
            \]
            Notice here that $b_i = \beta_{i, n+i}(I(\Delta, \mC))$ and $f_i = f_i(\Delta)$.
        \item The preceding comment is reminiscent of a result of Eagon and Reiner~\cite[Corollary~5]{ER}.  If $\Delta^{\vee}$
            is a pure shellable simplicial complex, then
            \[
                \sum_{i \geq 1} \beta_{i}(R/I_{\Delta}) x^i = \sum_{i \geq 0} h_i(\Delta^{\vee})(1 + x)^i.
            \]
            Here $h(\Delta^{\vee})$ is the $h$-vector of $\Delta^{\vee}$, a well-studied derivative vector of the $f$-vector.
    \end{enumerate}
\end{remark}

\section{Ferrers hypergraphs \& Boij-S\"oderberg decompositions}\label{sec:ferrers}

In this section, we connect $P(\Delta, \mC)$ with $k$-uniform Ferrers hypergraphs when $\mC$ is nested.
Through this connection, we use results of Nagel and Sturgeon~\cite{NS} to find the Boij-S\"oderberg
decomposition of both $I(\Delta, \mC)$ and $R/I(\Delta, \mC)$.

\subsection{Ferrers hypergraphs}\label{sub:ferrers}~

A \emph{hypergraph} $H$ is a vertex set $V$ together with a collection of edges $E$ that are
subsets of $V$.  In particular, notice that a simplicial complex is a special type of hypergraph.
If every edge of $H$ has the same cardinality, say, $k$, then $H$ is \emph{$k$-uniform}.  A
\emph{Ferrers hypergraph} is a $k$-uniform hypergraph $F$ on a vertex set $V_1 \ddd V_k$ such
that there is a linear ordering on each $V_j$ and whenever $(i_1, \ldots, i_k) \in F$ and
$(j_1, \ldots, j_k)$ is componentwise less than $(i_1, \ldots, i_k)$, then $(j_1, \ldots, j_k) \in F$.

Let $F$ be a Ferrers hypergraph $F$ on vertex set $V_1 \ddd V_k$.  Let
$R$ be the polynomial ring with variables $x_{j,i}$, where $i \in V_j$.  The
\emph{(generalised) Ferrers ideal of $F$} is the squarefree monomial ideal $I(F)$ of $R$
generated by monomials $\prod_{j=1}^{k} x_{j,i_j}$, where $(i_1, \ldots, i_k) \in F$.
These ideals have been studied extensively by Nagel and Reiner~\cite[Section~3]{NR}; in
particular, the $\ZZ$-graded Betti numbers of $I(F)$ are explicitly derived~\cite[Corollary~3.14]{NR}
from the construction of a minimal cellular resolution~\cite[Theorem~3.13]{NR}.

By construction, a $k$-uniform hypergraph is a Ferrers hypergraph if it is an order ideal of
$V_1 \ddd V_k$ under a componentwise partial order.  In particular, using
Proposition~\ref{pro:nested-order-ideal} we see that the set of $k$-uniform Ferrers hypergraphs on $n+k$
vertices is in bijection with the set of simplicial complexes on $n$ vertices together with nested
$k$-colourings of $\Delta$.  More explicitly, let $\Delta$ be a simplicial complex on $[n]$, and let
$\mC = C_1 \ddd C_k$ be a nested $k$-colouring of $\Delta$ with the nesting order.  In this case,
$P(\Delta, \mC)$ is a $k$-uniform Ferrers hypergraph on the vertex set $\{0, \ldots, \#C_1\} \ddd \{0, \ldots, \#C_k\}$.

Moreover, the uniform face ideal $I(\Delta, \mC)$ is squarefree precisely when the colour classes of
$\mC$ have cardinality at most one; see Section~\ref{sub:singleton} for a further discussion of such
colourings.  Thus, in particular, $I(\Delta, \mC) = I(P(\Delta, \mC))$ if and only if the colour
classes of $\mC$ have cardinality at most $1$.

\subsection{Boij-S\"oderberg decompositions}\label{sub:boij-soederberg}~

Let $R = K[x_1, \ldots, x_n]$ be standard graded, i.e., $\deg{x_i} = 1$, where $K$ is any field.
As in Section~\ref{sub:prelim-res}, we consider a minimal free resolution of a finitely generated
graded $R$-module $M$
\[
    0 \longrightarrow F_n \longrightarrow \cdots \longrightarrow F_0 \longrightarrow M \longrightarrow 0,
\]
where for $0 \leq i \leq n$ we have $F_i = \bigoplus_{j \in \ZZ} R(-j)^{\beta_{i,j}(M)}$.  The exponents
in the minimal free resolution of $M$ are expressed in the \emph{Betti table of $M$}, that is,
$\beta(M) = (\beta_{i,j}(M))$.

\begin{remark}\label{rem:betti-tables}
    We write the Betti table of $M$ as a $(\reg{M} + 1) \times (\pdim{M}+1)$ matrix with entry $(i,j)$
    given by $\beta_{i,i+j}(M)$.  For example, if $\beta_{0,0}(M) = 1$, $\beta_{1,3}(M) = 6$,
    $\beta_{2,4}(M) = 7$, and $\beta_{3,5}(M) = 2$, then we would write the Betti table of $M$ as
    \[
        \beta(M) = \begin{array}{c|cccc}
            \beta_{i,j} & 0 & 1 & 2 & 3 \\
            \hline
            0 & 1 & . & . & . \\
            1 & . & . & . & . \\
            2 & . & . & . & . \\
            3 & . & 6 & 7 & 2
        \end{array}
    \]
    Notice that we represent zeros by periods.  This follows the convention of the
    {\tt BettiTally} object in \emph{Macaulay2}~\cite{M2}.
\end{remark}

Boij and S\"oderberg~\cite[Conjecture~2.4]{BS-2008} conjectured a complete characterisation, up to multiplication
by a positive rational number, of the structure of Betti tables of finitely generated Cohen-Macaulay $R$-graded modules.
The conjecture was proven in characteristic zero by Eisenbud and Schreyer~\cite[Theorem~0.2]{ES} and in positive
characteristic by Eisenbud, Fl{\o}ystad, and Weyman~\cite{EFW}.  Furthermore, it was recently extended to the non-Cohen-Macaulay
case by Boij and S\"oderberg~\cite[Theorem~4.1]{BS-2012}.

We recall the characterisation here, where we follow the notation given in~\cite{BS-2012}, which differs from the
notation given with the original conjecture in~\cite{BS-2008}.  For an increasing sequence of integers
${\bf d} = (d_0, \ldots, d_s)$, where $0 \leq s \leq n$, the \emph{pure diagram given by ${\bf d}$} is the matrix
$\pi({\bf d})$ with entries given by
\[
    \pi({\bf d})_{i,j} =
    \begin{cases}
        \displaystyle (-1)^i \prod_{j=0, j\neq i}^{s} \frac{1}{d_j - d_i}, & \text{if~} j = d_i; \\
        0, & \text{otherwise.}
    \end{cases}
\]
Furthermore, we define a partial order on the pure diagrams by $\pi(d_0, \ldots, d_s) \leq \pi(d'_0, \ldots, d'_t)$,
if $s \geq t$ and $d_i \leq d'_i$ for $0 \leq i \leq t$.

The Boij-S\"oderberg characterisation is that the Betti table of a finitely generated $R$-module can be uniquely
decomposed into a linear combination of pure diagrams that form a chain in the partial order.  Specifically, we
have the following (reformulated as in~\cite[Theorem~2.2]{NS}) theorem.

\begin{theorem}{\cite[Theorem~4.1]{BS-2012}}\label{thm:BS-decomp}
    Let $R = K[x_1, \ldots, x_n]$ be standard graded, and let $M$ be a finitely generated graded $R$-module.
    There exists a unique chain of pure diagrams $\pi({\bf d}_0) < \ldots < \pi({\bf d}_t)$ and positive
    integers $a_0, \ldots, a_t$ such that
    \[
        \beta(M) = \sum_{j=0}^{t} a_j \pi({\bf d}_j).
    \]
\end{theorem}

We derive the Boij-S\"oderberg decomposition of the uniform face ideals coming from nested colourings.  In particular,
notice that the coefficients on the pure diagrams are completely determined by the $f$-vector of the simplicial complex.

\begin{proposition}\label{pro:bsd-I}
    Let $\Delta$ be a simplicial complex on $[n]$.  If $\mC$ is a nested $k$-colouring of $\Delta$
    endowed with the nesting order, then
    \[
        \beta(I(\Delta, \mC)) = \sum_{j = 0}^{1+\dim{\Delta}} j! f_{j-1}(\Delta) \pi(n, \ldots, n+j).
    \]
\end{proposition}
\begin{proof}
    For an integer $i$, the $(i, n+i)$ entry of $\pi(n, \ldots, n+j)$ is $\frac{1}{i!(j-i)!}$ if
    $j \geq i$ and $0$ otherwise.  Hence the $(i, n+i)$ entry of
    $\sum_{j = 0}^{1+\dim{\Delta}} j! f_{j-1}(\Delta) \pi(n, \ldots, n+j)$ is
    \[
        \sum_{j=i}^{j + \dim{\Delta}} j! f_{j-1}(\Delta) \frac{1}{i!(j-i)!} = \sum_{j=i}^{1 + \dim \Delta} \binom{j}{i} f_{j-1}(\Delta).
    \]
    Thus the claim follows by Theorem~\ref{thm:betti}.
\end{proof}

\begin{example}\label{exa:bs-I}
    Let $\Delta = \langle abc, bcd, ce, de, df \rangle$ be the simplicial complex given in Example~\ref{exa:nested},
    and let $\mC$ be any nested colouring of $\Delta$.  For example, if $\mC = \{d,a\} \dcup \{b,e\} \dcup \{c,f\}$,
    then $I(\Delta, \mC)$ is given in Example~\ref{exa:ufi}.

    Since $f(\Delta) = (1,6,8,2)$, the Boij-S\"oderberg decomposition of $\beta(I(\Delta,\mC))$ is
    \begin{equation*}
        \begin{split}
            \beta(I(\Delta, \mC))
               = & \begin{array}{c|cccc} \beta_{i,j} & 0 & 1 & 2 & 3 \\ \hline 6 & 17 & 28 & 14 & 2 \end{array} \\
               = & 3! \cdot f_{2}(\Delta) \cdot \pi(6,7,8,9) +
                   2! \cdot f_{1}(\Delta) \cdot \pi(6,7,8) \\
                 & + 1! \cdot f_{0}(\Delta) \cdot \pi(6,7) +
                   0! \cdot f_{-1}(\Delta) \cdot \pi(6) \\
               = & 3! \cdot 2 \cdot \begin{array}{c|cccc} \beta_{i,j} & 0 & 1 & 2 & 3 \\ \hline 6 & \frac{1}{6} & \frac{1}{2} & \frac{1}{2} & \frac{1}{6} \end{array} +
                   2! \cdot 8 \cdot \begin{array}{c|ccc} \beta_{i,j} & 0 & 1 & 2 \\ \hline 6 & \frac{1}{2} & 1 & \frac{1}{2} \end{array} \\
                 & + 1! \cdot 6 \cdot \begin{array}{c|cc} \beta_{i,j} & 0 & 1 \\ \hline 6 & 1 & 1 \end{array} +
                   0! \cdot 1 \cdot \begin{array}{c|c} \beta_{i,j} & 0 \\ \hline 6 & 1 \end{array}
        \end{split}
    \end{equation*}
\end{example}

\begin{remark}\label{rem:NS}
    By Theorem~\ref{thm:betti}, $\beta(I(\Delta, \mC))$ depends only on $\Delta$, when $\mC$ is a nested $k$-colouring.
    Thus we may assume $\mC = \{1\} \ddd \{n\}$ is the singleton colouring, and so $I(\Delta, \mC) = I(P(\Delta, \mC))$ is
    also a Ferrers ideal.  With this, the preceding result follows immediately from Nagel and Sturgeon's decomposition of
    the Betti table of a Ferrers ideal~\cite[Proposition~3.2]{NS}; specifically, we have shown that
    $\alpha(I(\Delta, \mC)) = f(\Delta)$.  It was using this result that Nagel and Sturgeon then classified the
    Boij-S\"oderberg decompositions of ideals with linear resolutions~\cite[Theorem~3.6]{NS} and hence the possible
    Betti numbers of ideals with linear resolutions.  The classification of Betti numbers of ideals with linear resolutions
    was first given by Murai~\cite[Proposition~3.8]{Mu}.
\end{remark}

Nagel and Sturgeon also classified the Boij-S\"oderberg decompositions of the quotients of Ferrers ideals~\cite[Theorem~3.9]{NS}.
However, this classification is complicated.  We recall it here as a convenience to the reader.

\begin{theorem}{\cite[Theorem~3.9]{NS}}\label{thm:ns-39}
    Let $k \geq 2$, and let $F$ be a $d$-uniform Ferrers hypergraph on the vertex set $V_1 \ddd V_d$, where each $V_i = [n_i]$ for some
    natural number $n_i$.  The Betti table of the quotient ring $R/I(F)$ is
    \[
        \beta(R/I(F)) = \sum_{j=1}^{d} \sum_{S \in F_j} n_S \cdot k_S! \cdot \pi(0, d, \ldots, d+k_S),
    \]
    where $F_j$ is the Ferrers hypergraph
    \[
        F_j := \{(i_1, \ldots, \widehat{i_j}, \ldots, i_d) \st \text{there is some } i_j \in V_j \text{ such that } (i_1, \ldots, i_j, \ldots, i_d) \in F\},
    \]
    and for each $S = (i_1, \ldots, \widehat{i_j}, \ldots, i_d) \in F_j$
    \[
        n_S := \max\{i_j \in V_j \st (i_1, \ldots, i_j, \ldots, i_d) \in F\} \text{~and~} k_S := n_S - d + \sum_{p = 1, p \neq j}^{d} i_p.
    \]
\end{theorem}

While the coefficients on the Boij-S\"oderberg decomposition of $\beta(R/I(\Delta, \mC))$ are complicated
compared to the coefficients on the Boij-S\"oderberg decomposition of $\beta(I(\Delta, \mC))$, they are
yet more simple than the coefficients in the preceding result.  Indeed, an explicit formul\ae\ for each coefficient
can be given in terms of the $f$-vectors of $\Delta$ and its links by vertices.

\begin{proposition}\label{pro:bsd-RI}
    Let $\Delta$ be a simplicial complex on $[n]$.  If $\mC$ is a nested $k$-colouring of $\Delta$
    endowed with the nesting order, then
    \[
        \beta(R/I(\Delta, \mC)) = \sum_{j = 1}^{1+\dim{\Delta}} a_j \pi(0, n, \ldots, n+j),
    \]
    where for $1 \leq j \leq 1 + \dim{\Delta}$ we have
    \[
        a_j = j! \left( n \cdot f_{j-1}(\Delta) +
            \sum_{v=1}^{n} \left(
                f_{j-2}(\link_{\Delta} v) - f_{j-1}(\link_{\Delta} v)
            \right)
        \right).
    \]
\end{proposition}
\begin{proof}
    By Theorem~\ref{thm:betti}, $\beta(I(\Delta, \mC))$ depends only on $\Delta$, when $\mC$ is a nested $k$-colouring.
    Thus we may assume $\mC = \{1\} \ddd \{n\}$ is the singleton colouring, and so $I(\Delta, \mC) = I(P(\Delta, \mC))$
    is also a Ferrers ideal.  Hence we can use Theorem~\ref{thm:ns-39} (i.e., \cite[Theorem~3.9]{NS}) to determine the
    Boij-S\"oderberg decomposition of $R/I(\Delta, \mC)$.

    Set $F = P(\Delta, \mC)$, and notice that $V_i = \{0,1\}$ for $1 \leq i \leq n$.  Due to the shift
    in the vertex set, we must compute $n_S$ as $1 + \max\{i_j \in V_j \st (i_1, \ldots, i_j, \ldots, i_n) \in F\}$
    and $k_S$ as $n_S - 1 + \sum_{p = 1, p \neq j}^{n} i_p$.

    For each $v \in [n]$, we have that $F_v = P(\Delta \setminus \{v\}, \mC \setminus \{v\})$.  Let $S \in F_v$; hence
    $S$ is associated to some face $\sigma$ of $\Delta \setminus \{v\}$.  If $\sigma \in \link_{\Delta} v$, then
    $\max\{i_v \in V_v \st (i_1, \ldots, i_v, \ldots, i_n) \in F\} = 1$, and so $n_S = 2$; otherwise, $n_S = 1$.
    Further still, we notice that $k_S = n_S + \dim{\sigma}$, since $\sum_{p = 1, p \neq v}^{n} i_p$ is the cardinality
    of $\sigma$.  Hence $\dim{\sigma} = k_S - 2$ if $\sigma \in \link_{\Delta} v$; otherwise $\dim{\sigma} = k_S - 1$.

    Let $j$ be a positive integer.  The coefficient on the pure diagram $\pi(0, n, \ldots, n+j)$ in
    the Boij-S\"oderberg decomposition of $R/I(F)$ is $j!$ times the sum of the number of dimension $j-1$ faces in each
    of the $\Delta \setminus \{v\}$ that are not in $\link_{\Delta} v$ together with the number of dimension $j-2$ faces
    in each of the $\link_{\Delta} v$, that is,
    \[
        j! \sum_{v \in \Delta} \left(
            f_{j-1}(\Delta \setminus \{v\}) - f_{j-1}(\link_{\Delta} v) + 2 \cdot f_{j-2}(\link_{\Delta} v)
        \right).
    \]
    However, we recall that $f_{j-1}(\Delta) = f_{j-1}(\Delta \setminus \{v\}) + f_{j-2}(\link_{\Delta} v)$ for any vertex
    $v \in \Delta$.  Hence we can simplify the above coefficient to the claimed form of $a_j$.
\end{proof}

\begin{example}\label{exa:bs-RI}
    Let $\Delta = \langle abc, bcd, ce, de, df \rangle$ be the simplicial complex given in Example~\ref{exa:nested},
    and let $\mC$ be any nested colouring of $\Delta$.  For example, if $\mC = \{d,a\} \dcup \{b,e\} \dcup \{c,f\}$,
    then $I(\Delta, \mC)$ is given in Example~\ref{exa:ufi}.  The Boij-S\"oderberg decomposition of $\beta(I(\Delta, \mC))$
    is given in Example~\ref{exa:bs-I}.

    In order to compute the Boij-S\"oderberg decomposition of $\beta(R/I(\Delta, \mC))$, we need to compute the
    $f$-vectors of the links of $\Delta$ by each of its vertices.  The desired $f$-vectors are
    \[
        f(\Delta) = (1,6,8,2), f(\link_{\Delta}(a)) = (1,2,1), f(\link_{\Delta}(b)) = (1,3,2), f(\link_{\Delta}(c)) = (1,4,2),
    \]
    \[
        f(\link_{\Delta}(d)) = (1,4,1), f(\link_{\Delta}(e)) = (1,2), \text{~and~} f(\link_{\Delta}(f)) = (1,1).
    \]
    Thus we compute the coefficients $a_j$, for $1 \leq j \leq 3$, as
    \begin{equation*}
        \begin{split}
            a_1 &= 1!(6 \cdot 6 + (1 - 2) + (1 - 3) + (1 - 4) + (1 - 4) + (1 - 2) + (1 - 1)) = 26, \\
            a_2 &= 2!(6 \cdot 8 + (2 - 1) + (3 - 2) + (4 - 2) + (4 - 1) + (2 - 0) + (1 - 0)) = 116, \text{and} \\
            a_3 &= 3!(6 \cdot 2 + (1 - 0) + (2 - 0) + (2 - 0) + (1 - 0) + (0 - 0) + (0 - 0)) = 108.
        \end{split}
    \end{equation*}
    Hence the Boij-S\"oderberg decomposition of $\beta(R/I(\Delta, \mC))$ is
    \begin{equation*}
        \begin{split}
            \beta(R/I(\Delta, \mC))
               = & \begin{array}{c|ccccc}
                        \beta_{i,j} & 0 & 1 & 2 & 3 & 4 \\
                        \hline
                        0 & 1 & .  & .  & .  & . \\
                        6 & . & 17 & 28 & 14 & 2 \\
                    \end{array} \\
              = & a_3 \cdot \pi(0,6,7,8,9) + a_2 \cdot \pi(0,6,7,8) + a_1 \cdot \pi(0,6,7) \\
              = & 108 \cdot
                    \begin{array}{c|ccccc}
                        \beta_{i,j} & 0 & 1 & 2 & 3 & 4 \\
                        \hline
                        0 & \frac{1}{3024} & . & . & . & . \\
                        6 & . & \frac{1}{36} & \frac{1}{14} & \frac{1}{16} & \frac{1}{54} \\
                    \end{array}
                 + 116 \cdot
                    \begin{array}{c|cccc}
                        \beta_{i,j} & 0 & 1 & 2 & 3  \\
                        \hline
                        0 &\frac{1}{336} & . & . & .  \\
                        6 & . & \frac{1}{12} & \frac{1}{7} & \frac{1}{16} \\
                    \end{array} \\
                & + 26 \cdot
                    \begin{array}{c|cccc}
                        \beta_{i,j} & 0 & 1 & 2 \\
                        \hline
                        0 & \frac{1}{42} & . & .   \\
                        6 & . & \frac{1}{6} & \frac{1}{7} \\
                    \end{array}
        \end{split}
    \end{equation*}
    Notice that the zero rows are suppressed from the Betti tables.
\end{example}

\section{Algebraic properties}\label{sec:properties}

In this section, we classify several algebraic properties of $I(\Delta, \mC)$, when $\mC$ is nested.
We first derive explicit formul\ae\ for several algebraic properties of $I(\Delta, \mC)$.  We then
describe the associated primes of $I(\Delta, \mC)$, which are persistent.

\subsection{Derivative properties}\label{sub:derivative}~

From the $\ZZ$-graded Betti numbers of $I(\Delta, \mC)$, which are given in Theorem~\ref{thm:betti}, we
find exact formul\ae\ for several algebraic properties of $I(\Delta, \mC)$.  The following results do not depend
on the field $K$, as the Betti numbers do not.

For a simplicial complex $\Delta$, we note that $e(K[\Delta]) = f_{d-1}$ and $\dim{K[\Delta]} = \dim{\Delta}+1$.
We explicitly give the $Q$-polynomial (and hence dimension, codimension, and multiplicity) of $R/I(\Delta, \mC)$,
when $\mC$ is nested.

\begin{theorem}\label{thm:dim-and-friends}
    Let $\Delta$ be a simplicial complex on $[n]$, and let $\mC$ be a nested $k$-colouring of $\Delta$
    endowed with the nesting order.  If $\dim{\Delta} \geq 0$, then the Hilbert series
    of $R/I(\Delta, \mC)$ is
    \[
        H_{R/I(\Delta, \mC)}(t) = \frac{\sum_{i=0}^{n-1}(i+1)t^i - t^n\sum_{i=2}^{k}f_{i-1}(\Delta)(1-t)^{i-2}}{(1-t)^{2k-2}}.
    \]
    In particular, $\codim{R/I(\Delta, \mC)} = 2$, $\dim{R/I(\Delta, \mC)} = 2k-2$, and $e(R/I(\Delta, \mC)) = \binom{n+1}{2} - f_{1}(\Delta)$.
\end{theorem}
\begin{proof}
    By Theorem~\ref{thm:betti}, the Poincar\'e polynomial of $R/I(\Delta, \mC)$ is
    \[
        1 - t^n \sum_{i=0}^{1 + \dim{\Delta}} f_{i-1}(\Delta)(1-t)^i.
    \]
    At $t = 1$, this reduces to $1 - f_{-1}(\Delta) = 0$, and so the polynomial is divisible by $(1-t)$.  Dividing
    by $(1-t)$ we obtain
    \[
        \sum_{i=0}^{n-1}t^i - t^n \sum_{i=1}^{1 + \dim{\Delta}} f_{i-1}(\Delta)(1-t)^{i-1}.
    \]
    Again, at $t = 1$, this reduces to $n - f_0(\Delta) = 0$ as $f_0(\Delta) = n$, and so the
    polynomial is divisible by $(1-t)$.  Dividing by $(1-t)$ we obtain
    \[
        \sum_{i=0}^{n-1}(i+1)t^i - t^n\sum_{i=2}^{k}f_{i-1}(\Delta)(1-t)^{i-2}.
    \]
    At $t = 1$, this reduces to $\binom{n+1}{2} - f_1(\Delta) > 0$ as $f_1(\Delta) \leq \binom{n}{2}$.
    Thus, in particular, the preceding polynomial is $Q_{R/I(\Delta, \mC)}(t)$, and so
    $e(I(\Delta, \mC)) = \binom{n+1}{2} - f_1(\Delta)$.  Further still, as $(1-t)$ divides
    the Poincar\'e polynomial precisely twice, we have $\codim{R/I(\Delta, \mC)} = 2$ and so
    $\dim{R/I(\Delta, \mC)} = 2k - 2$.
\end{proof}

The projective dimension of $R/I(\Delta, \mC)$ is a constant translation of $\dim{\Delta}$, as is $\dim{K[\Delta]}$.

\begin{corollary}\label{cor:pdim}
    Let $\Delta$ be a simplicial complex on $[n]$.  If $\mC$ is a nested $k$-colouring of $\Delta$
    endowed with the nesting order, then $\pdim{R/I(\Delta, \mC)} = \dim{\Delta}+2$.  Moreover, if
    $\mC$ has $t$ trivial colour classes, then $\pdim{R/I(\Delta, \mC)^i}$ weakly increases
    to $k-t+1$, with equality guaranteed if $i \geq k-t-\dim{\Delta}-1$.
\end{corollary}
\begin{proof}
    It follows immediately from Theorem~\ref{thm:betti} that $\pdim{R/I(\Delta, \mC)} = \dim{\Delta}+2$.

    Suppose $I(\Delta, \mC)^i = I(\Gamma, \mD)$.  Clearly, $\Delta \subset \Gamma$ and so
    $\dim{\Delta} \leq \dim{\Gamma}$, thus $\pdim{R/I(\Delta, \mC)^i}$ weakly increases.  Let
    $\sigma$ be a $(\dim{\Delta})$-dimensional face of $\Delta$.  Since $\sigma$ intersects
    $\dim{\Delta} + 1$ colour classes of $\mC$ nontrivially, and $t$ colour classes of $\mC$ are
    trivial, then there are $j = k - t - \dim{\Delta} - 1$ nontrivial colour classes of $\mC$ that
    do not intersect $\sigma$.  Pick vertices $v_1, \ldots, v_j$ from each of these $j$ colour classes.
    In this case, $m_{\sigma} \cdot m_{v_1} \cdots m_{v_j}$ is a member of $I(\Delta, \mC)^{k-t-\dim{\Delta}-1}$,
    that is, $\sigma \dcup \{v_1, \ldots, v_j\}$ is a face of $\Gamma$.  Since $\mC$ has precisely
    $k - t$ nontrivial colour classes, $\mD$ has precisely $k - t$ nontrivial colour classes.
    Thus $\dim{\Gamma} = k - t - 1$, and so $\pdim{R/I(\Delta, \mC)^i} \leq k-t +1$.
\end{proof}

Hence the depth of $R/I(\Delta, \mC)$ is a linear translation of $\dim{\Delta}$.

\begin{corollary}\label{cor:depth}
    Let $\Delta$ be a simplicial complex on $[n]$.  If $\mC$ is a nested $k$-colouring of $\Delta$
    endowed with the nesting order, then $\depth{R/I(\Delta, \mC)} = 2(k-1) - \dim{\Delta} \geq \dim{\Delta}$.
    Moreover, if $\mC$ has $t$ trivial colour classes, then $\depth{R/I(\Delta, \mC)^i}$ weakly decreases
    to $k+t-1$, with equality guaranteed if $i \geq k-t-\dim{\Delta}-1$.
\end{corollary}
\begin{proof}
    This follows from Corollary~\ref{cor:pdim} and the Auslander-Buchsbaum formula.
\end{proof}

Thus a uniform face ideal is Cohen-Macaulay if and only if the simplicial complex is zero-dimensional.
This implies that a higher power of a Cohen-Macaulay uniform face ideal is Cohen-Macaulay if and only if
the initial colouring has exactly one nontrivial colour class.

\begin{corollary}\label{cor:CM}
    Let $\Delta$ be a simplicial complex on $[n]$, and let $\mC$ be a nested $k$-colouring of $\Delta$
    endowed with the nesting order.  The quotient $R/I(\Delta, \mC)$ is Cohen-Macaulay if and only
    if $\dim{\Delta} = 0$.
\end{corollary}
\begin{proof}
    By Corollary~\ref{cor:depth}, we have that $\depth{R/I(\Delta, \mC)} = 2(k-1) - \dim{\Delta}$, and by
    Theorem~\ref{thm:dim-and-friends}, we have that $\dim{R/I(\Delta, \mC)} = 2k-2$.  Hence
    $\depth{R/I(\Delta, \mC)} = \dim{R/I(\Delta, \mC)}$ if and only if $2(k-1) - \dim{\Delta} = 2k-2$,
    i.e., $\dim{\Delta} = 0$.
\end{proof}

Cutkosky, Herzog, and Trung~\cite[Theorem~1.1]{CHT} showed that the regularity of powers of a homogeneous
ideal is eventually a linear function; this was independently shown by Kodiyalam~\cite[Theorem~5]{Ko}.
The regularity of a uniform face ideal is the number of vertices of $\Delta$, and hence the regularity
of powers of a uniform face ideal is linear from the start.

\begin{corollary}\label{cor:regularity}
    Let $\Delta$ be a simplicial complex on $[n]$.  If $\mC$ is a nested $k$-colouring of $\Delta$
    endowed with the nesting order, then $\reg{R/I(\Delta, mC)^i} = i \cdot n$.
\end{corollary}
\begin{proof}
    This follows immediately from Theorem~\ref{thm:betti}.
\end{proof}

\subsection{Irreducible decompositions \& associated primes}\label{sub:associated}~

Let $\Delta$ be a simplicial complex on $[n]$, and let $\mC = C_1 \ddd C_k$ be a proper vertex $k$-colouring
of $\Delta$.  In this case, the minimal non-faces of $\Delta$ are either contained in a colour class
$C_i$ or intersect each colour class in at most one vertex.  Thus we can similarly define $e(\sigma)$ for the
latter variety of non-faces of $\Delta$.  Let $\mN(\Delta, \mC)$ be the set of index vectors of minimal non-faces
of $\Delta$ that are not contained in a single colour class.  Clearly, $\mN(\Delta, \mC)$ is a poset under the
componentwise partial order.  Further, we note that $\mN(\Delta, \mS)$, where $\mS$ is the singleton colouring
of $\Delta$, is an antichain labeled by the index vectors of all of the minimal non-faces of $\Delta$.

\begin{example}\label{exa:nonface-poset}
    Let $\Delta = \langle abc, bcd, ce, de, df \rangle$ be the simplicial complex given in Example~\ref{exa:nested}.
    The minimal non-faces of $\Delta$ are $\{ad, ae, af, be, bf, cde, cf, ef\}$.  Given the nested colouring
    $\mC = \{d,a\} \dcup \{b,e\} \dcup \{c,f\}$, also from Example~\ref{exa:nested}, the only minimal non-faces
    that are necessary to track are $\{ae, af, bf, cde, ef\}$, as the non-faces $\{ad, be, cf\}$ are all given by
    the colouring itself.  Thus the minimal non-face poset $\mN(\Delta, \mC)$ has five vertices.
    \begin{figure}[!ht]
        \centering
        \includegraphics[scale=1.75]{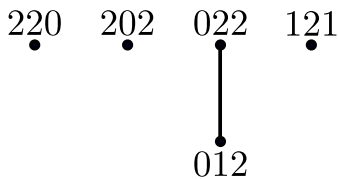}
        \caption{The minimal non-face poset $\mN(\Delta, \mC)$, where $\Delta = \langle abc, bcd, ce, de, df \rangle$ and
            $\mC = \{d,a\} \dcup \{b,e\} \dcup \{c,f\}$, as in Example~\ref{exa:nested}.}
        \label{fig:Delta-N}
    \end{figure}
\end{example}

Let $R = K[x_1, \ldots, x_n]$.  A monomial ideal $I$ of $R$ is \emph{irreducible} if it is of the form
$\mm^{\bf b} = (x_i^{b_i} \st b_i \geq 1)$, where ${\bf b} = (b_1, \ldots, b_n) \in \NN_0^n$.  For any
monomial ideal $I$ of $R$, an \emph{irreducible decomposition of $I$} is an expression of the form
$I = \mm^{\bf b_1} \cap \cdots \cap \mm^{\bf b_t}$, for ${\bf b_1}, \ldots, {\bf b_t} \in \NN_0^n$.
Such a decomposition is \emph{irredundant} if none of the components can be omitted.  We note that
every monomial ideal has a unique irredundant irreducible decomposition (see, e.g., \cite[Theorem~1.3.1]{HH}
or~\cite[Theorem~5.27]{MS}).

We can exploit the structure of uniform face ideals coming from nested colourings to determine an
irredundant irreducible decomposition.

\begin{lemma}\label{lem:decomp}
    Let $\Delta$ be a simplicial complex on $[n]$, and let $\mC = C_1 \ddd C_k$ be a nested $k$-colouring
    of $\Delta$ endowed with the nesting order.  If $\mN' = \min \mN(\Delta, \mC)$ is the set of minimal
    elements of $\mN(\Delta, \mC)$, then the intersection
    \[
        \left(
            \bigcap_{i=1}^{k} \left(
                \bigcap_{j=1}^{\#C_i} (x_i^j, y_i^{\#C_i - j + 1})
            \right) 
        \right)
        \cap
        \left(
            \bigcap_{e \in \mN'} \left( x_j^{\#C_j - e_j + 1} \st e_j > 0 \right)
        \right)
    \]
    is the irredundant irreducible decomposition of $I(\Delta, \mC)$.
\end{lemma}
\begin{proof}
    Set $I'$ to be the intersection in the claim.  Each factor of $I'$ is generated by pure powers
    of variables, hence each factor is irreducible.  Clearly, the factors of $I'$ of the form
    $(x_i^j, y_i^{\#C_i- j+1})$ are non-redundant factors of $I'$.  For each $e \in \mN'$, set
    $\mm_e = ( x_j^{\#C_j - e_j + 1} \st e_j > 0)$.  Since $e$ must be nonzero in at least two
    positions, the $\mm_e$ are non-redundant with the factors of the form $(x_i^j, y_i^{\#C_i- j+1})$.

    Let $e, e' \in \mN'$ such that $\mm_e \subset \mm_e'$.  This implies that $\#C_j - e_j + 1 \geq \#C_j - e'_j + 1$,
    for $1 \leq j \leq k$, i.e., $e' \geq e$ under a componentwise partial order.  Since $e$ and $e'$
    are both minimal elements of $\mN(\Delta, \mC)$, $e = e'$.  Hence the factors of the form $\mm_e$
    are all non-redundant.  Thus the presentation of $I'$ is an irredundant irreducible decomposition.

    For $1 \leq i \leq k$, we have $(x_i, y_i)^{\#C_i} = \bigcap_{j=1}^{\#C_i} (x_i^j, y_i^{\#C_i - j + 1})$.
    By construction, every monomial of $I(\Delta, \mC)$ has degree $\#C_i$ in the variables $x_i$ and $y_i$;
    hence $I(\Delta, \mC) \subset (x_i, y_i)^{\#C_i}$ for $1 \leq i \leq k$.  Let $m_\tau$ be a minimal generator
    of $I(\Delta, \mC)$, i.e., $\tau$ is a face of $\Delta$.  Suppose $m_\tau \notin \mm_{\sigma}$ for some
    $\sigma \in \mN'$.  This implies that $\#C_i - e_i(\sigma) + 1 > \#C_i - e_i(\tau)$, for all $1 \leq i \leq k$.
    Hence $e_i(\tau) \geq e_i(\sigma)$; this contradicts Proposition~\ref{pro:nested-order-ideal}.  Thus
    $m_\tau \in \mm_{e(\sigma)}$ for all $e(\sigma) \in \mN'$, i.e., $I(\Delta, \mC) \subset I'$.

    Let $g$ be a monomial of $I'$.  Since $g \in (x_i, y_i)^{\#C_i}$ for $1 \leq i \leq k$, $m_\tau$ divides
    $g$ for some $\tau \in 2^{[n]}$; note that $\tau$ intersects each colour class in at most one vertex.
    Since $g \in \mm_{\sigma}$ for all $e(\sigma) \in \mN'$, we further have that $e_j(\sigma) > e_j(\tau)$
    for some $1 \leq j \leq k$.  Since the $\sigma$ are minimal non-faces of $\Delta$, $\tau$ must be a face
    of $\Delta$.  Hence $m_\tau \in I(\Delta, \mC)$, and so $g \in I(\Delta, \mC)$.  Thus $I' \subset I(\Delta, \mC)$.
\end{proof}

Let $R = K[x_1,\ldots,x_n]$ be the $n$-variate polynomial ring over a field $K$, and let $I$ be an ideal of $R$.
A prime ideal $\mathfrak{p}$ of $R$ is an \emph{associated prime of $R/I$} if there exists an $a \in R$ such
that $I:(a) = \{f \in R \st af \in I\} = \mathfrak{p}$.  The set of associated primes of $R/I$ is denoted $\ass(R/I)$.
For any ${\bf b} \in \NN_0^n$, we have $\ass(R/\mm^{\bf b}) = \{(x_i \st b_i \geq 1)\}$.  It follows that
if $I = \mm^{\bf b_1} \cap \cdots \cap \mm^{\bf b_t}$ is the irredundant irreducible decomposition of the monomial
ideal $I$, then $\ass(R/I) = \{(x_i \st b_{1,i} \geq 1), \ldots, (x_i \st b_{t,i} \geq 1)\}$ (see the comments
following~\cite[Proposition~1.3.7]{HH}).

As an immediate corollary of the preceding lemma, we can classify the associated primes of $I(\Delta, \mC)$.

\begin{corollary}\label{cor:assoc}
    Let $\Delta$ be a simplicial complex on $[n]$, and let $\mC = C_1 \ddd C_k$ be a nested $k$-colouring
    of $\Delta$ endowed with the nesting order.  If $\mN' = \min \mN(\Delta, \mC)$ is the set of minimal
    elements of $\mN(\Delta, \mC)$, then $\ass(R/I(\Delta, \mC))$ is
    \[
        \left\{ (x_1, y_1), \ldots, (x_k, y_k), (x_j \st 1 \leq j \leq k, 0 < e_j) \text{~for each~} e \in \mN' \right\}.
    \]

    In particular, if $\mS = \{1\} \ddd \{n\}$ is the singleton colouring, then $\ass(R/I(\Delta, \mS))$ is
    \[
        \left\{ (x_1, y_1), \ldots, (x_n, y_n), (x_j \st j \in \sigma) \text{~for each minimal non-face~} \sigma \notin \Delta \right\}.
    \]
\end{corollary}

Since the minimal primes are a subset of the associated primes, the preceding corollary implies that uniform face ideals
coming from flag complexes endowed with a nested colouring are unmixed.

\begin{corollary}\label{cor:unmixed}
    Let $\Delta$ be a simplicial complex on $[n]$, and let $\mC = C_1 \ddd C_k$ be a nested $k$-colouring
    of $\Delta$ endowed with the nesting order.  If $\Delta$ is flag, then $I(\Delta, \mC)$ is unmixed.
\end{corollary}

Non-flag complexes may also yield unmixed uniform face ideals; see Example~\ref{exa:persistence}.  However,
as the associated primes of a squarefree ideal are the minimal primes, this implies that the squarefree uniform
face ideals are unmixed precisely when the complex is flag.

\begin{corollary}\label{cor:sqfree-unmixed}
    Let $\Delta$ be a simplicial complex on $[n]$.  If $\mS = \{1\} \ddd \{n\}$ is the singleton colouring of $\Delta$,
    then $R/I(\Delta, \mS)$ is unmixed if and only if $\Delta$ is a flag complex.
\end{corollary}

\begin{remark}\label{rem:Hibi-ideals}
    Let $Q$ be a poset on $[n]$.  The \emph{Hibi ideal of $Q$} is the squarefree monomial ideal $H_Q$ in
    $R[x_1, \ldots, x_n, y_1, \ldots, y_n]$ generated by the monomials $u_I = \left( \prod_{i \in I} x_i \right) \cdot \left( \prod_{i \notin I} y_i \right)$,
    for any order ideal $I$ of $Q$.  Herzog and Hibi~\cite[Lemma~3.1]{HH-p} showed that $H_Q = I_{\ind{B_Q}^{\vee}}$,
    where $B_Q$ is the comparability bi-graph of $Q$, i.e., $V(B_Q) = \{x_1, \ldots, x_n, y_1, \ldots, y_n\}$
    and $(x_i, y_j) \in E(B_Q)$ if $i \leq_Q j$.  Every power of $H_Q$ has a linear resolution~\cite[Corollary~1.3]{HH-p},
    the multiplicity of $H_Q$ is $\#\{i \leq_Q j \st 1 \leq i,j \leq n\}$~\cite[Proposition~2.4]{HH-p}, and the irredundant
    irreducible decomposition of $H_Q$ is $\bigcap_{i \leq_Q j} (x_i, y_j)$~\cite[Comment following Corollary~2.3]{HH-p}.

    By looking at a slightly different graph, we find a uniform face ideal that has similar properties.
    Let $G_Q$ be the comparability graph of $Q$, i.e., $V(G_Q) = [n]$ and $(i,j) \in E(G_Q)$ if $i <_Q j$ or $j <_Q i$.  The independence
    complex $\ind{G_Q}$ is generated by the antichains of $Q$.  By Corollary~\ref{cor:assoc}, we have
    \[
        I(\ind{G_Q}, \mS) = \left(\bigcap_{i=1}^n (x_i, y_i) \right) \cap \left (\bigcap_{i <_Q j} (x_i, x_j) \right).
    \]
    Moreover, following Theorem~\ref{thm:dim-and-friends}, we have
    \[
        e(R/I(\Delta_Q, \mS)) = e(R/H_Q) = \#\{i \leq_Q j \st 1 \leq i,j \leq n\}.
    \]
\end{remark}

An ideal $I$ of $R$ has \emph{persistent associated primes} if $\ass(R/I^{i}) \subset \ass(R/I^{i+1})$ for
$i \geq 1$.  Not all monomial ideals have persistent associated primes; Herzog and Hibi~\cite{HH-2005} gave
such an example.  Further, not all squarefree monomial ideals have persistent associated primes, as shown by
Kaiser, Stehl\'{\i}k, and \v{S}krekovski~\cite{KSS}.  However, uniform face ideals do have persistent associated primes
if they are derived from nested colourings.

\begin{theorem}\label{thm:persistence}
    Let $\Delta$ be a simplicial complex on $[n]$, and let $\mC = C_1 \ddd C_k$ be a nested $k$-colouring
    of $\Delta$ endowed with the nesting order.  If $i \geq 1$, then $\ass(R/I(\Delta, \mC)^{i}) \subset \ass(R/I(\Delta, \mC)^{i+1})$.
\end{theorem}
\begin{proof}
    Let $\Gamma$ be the simplicial complex on $[i \cdot n]$ with nested $k$-colouring $\mD$ such that
    $I(\Delta, \mC)^{i} = I(\Gamma, \mC)$.  Define $\Sigma$ and $\mE$ similarly for $I(\Delta, \mC)^{i+1}$.

    Let $e$ be a minimal element of $\mN(\Gamma, \mD)$.  For each $j$ such that $e_j > 0$, the vector
    $(e_1, \ldots, e_j - 1, \ldots, e_k)$ is member of $P(\Gamma, \mD)$, as $e$ is minimal.  The latter implies that
    $e \in P(\Sigma, \mE)$.  Let $c = (c_1, \ldots, c_k)$ be given by $c_j = 0$ if $e_j = 0$ and $c_j = \#C_j$ if $e_j > 0$.
    As $e \leq c$ under the componentwise order, $c \notin P(\Delta, \mC)$, and thus $(i+1) \cdot c \notin P(\Sigma, \mE)$.
    Since $e \in P(\Sigma, \mE)$, $(i+1) \cdot c \notin P(\Sigma, \mE)$, and $e < (i+1) \cdot c$, there must exist a
    minimal element $f$ of $\mN(\Sigma, \mE)$ such that $e < f \leq (i+1) \cdot c$.  Since $e_j = 0$ precisely when $f_j = 0$,
    $\ass(R/I(\Gamma, \mD)) \subset \ass(R/I(\Sigma, \mE))$.
\end{proof}

In general, equality need not hold in the containment of the preceding theorem.

\begin{example}\label{exa:persistence}
    Recall the simplicial complex $\Delta = \langle abc, bcd, ce, de, df \rangle$ and the nested colouring
    $\mC = \{d,a\} \dcup \{b,e\} \dcup \{c,f\}$ of $\Delta$ given in Example~\ref{exa:nested}.  By Example~\ref{exa:nonface-poset},
    we see that the minimal elements of $\mN(\Delta, \mC)$ are $(2,2,0)$, $(2,0,2)$, $(0,2,2)$, and $(1,2,1)$.  Thus
    by Corollary~\ref{cor:assoc}, the associated primes of $R/I(\Delta, \mC)$ are
    \[
        \left\{ (x_1, y_1), (x_2, y_2), (x_3, y_3), (x_1, x_2), (x_1, x_3), (x_2, x_3), (x_1, x_2, x_3) \right\}.
    \]
    As all possible prime ideals for a uniform face ideal coming from a nested $3$-colouring are present, the associated primes of
    all powers of $I(\Delta, \mC)$ are the same.

    Consider $\mD = \{d,a\} \dcup \{b,e\} \dcup \{c\} \dcup \{f\}$.  In this case, the minimal elements of
    $\mN(\Delta, \mD)$ are $(2, 2, 0, 0)$, $(2,0,0,1)$, $(0,1,0,1)$, $(0,0,1,1)$, and $(1,2,1,0)$.  Thus
    by Corollary~\ref{cor:assoc}, the associated primes of $R/I(\Delta, \mD)$ are
    \[
        \left\{ (x_1, y_1), (x_2, y_2), (x_3, y_3), (x_4, y_4), (x_1, x_2), (x_1, x_4), (x_2, x_4), (x_3, x_4), (x_1, x_2, x_3) \right\}.
    \]
    Furthermore, the associated primes of $R/I(\Delta, \mD)^2$ are
    \[
        \ass(R/I(\Delta, \mD)) \dcup \left\{ (x_1, x_2, x_4), (x_1, x_2, x_3, x_4) \right\}.
    \]
    The two new associated primes come from the minimal elements $(3, 2, 0, 1)$ and $(2, 2, 1, 1)$.  Experimentally, the associated
    primes of powers of $I(\Delta, \mD)$ are stable after the second power.

    We note that the ideals above are all unmixed despite $\Delta$ being non-flag (cf.\ Corollary~\ref{cor:sqfree-unmixed}).
\end{example}

\begin{acknowledgement}
    We thank Juan Migliore, Uwe Nagel, Stephen Sturgeon, and Adam Van Tuyl for many insightful comments.
    We also acknowledge the extensive use of \emph{Macaulay2}~\cite{M2} throughout this project.
\end{acknowledgement}


\end{document}